\documentclass[a4paper]{article}

\usepackage{authblk}

\usepackage{silence}
\usepackage[UKenglish]{babel}
\usepackage[style=alphabetic,maxbibnames=99]{biblatex}
\usepackage{amsmath}
\usepackage{mathtools}
\usepackage{amsthm}
\usepackage{thmtools}
\usepackage{csquotes}
\usepackage[inline]{enumitem}
\usepackage{nicefrac}
\usepackage{fonttable}
\usepackage{tikz}
\usepackage[hidelinks]{hyperref}
\usepackage[capitalize,noabbrev]{cleveref}
\usepackage{newunicodechar}
\usepackage[final]{microtype}

\usepackage[charter]{mathdesign}
\usepackage[T1]{fontenc}

\newtheorem{theorem}{Theorem}
\newtheorem{lemma}[theorem]{Lemma}
\newtheorem{proposition}[theorem]{Proposition}
\newtheorem{corollary}[theorem]{Corollary}
\newtheorem{question}[theorem]{Question}

\theoremstyle{definition}
\newtheorem{definition}[theorem]{Definition}

\newtheorem{example}[theorem]{Example}
\theoremstyle{remark}
\newtheorem{remark}[theorem]{Remark}

\usetikzlibrary{
	cd,
	positioning,
	arrows,
 	arrows.meta,
	calc,
	external,
	math,
	decorations.pathmorphing,
	decorations.markings,
	calligraphy,
	patterns,
	patterns.meta,
	graphs
}
\tikzset{
	point/.style={draw=black,fill=black,opacity=1,circle,outer sep=0pt,inner sep=0,minimum size=2},
	dot/.style={draw=black,fill=black,opacity=0,circle,outer sep=0pt,inner sep=0}
}
\tikzgraphsset{
	poset/.style={grow up=#1, branch right=#1, empty nodes, simple, nodes=point},
	poset/.default=1cm
}
\tikzset{
	negated/.style={
        decoration={markings,
            mark= at position 0.5 with {
                \node[transform shape] (tempnode) {$\backslash$};
            }
        },
        postaction={decorate}
    }
}
\tikzset{
	brace/.style={decoration={calligraphic brace,amplitude=5pt}, decorate, line width=1.25pt}
}

\definecolor{pallette_orange}{RGB}{230, 159, 0}
\definecolor{pallette_skyblue}{RGB}{86, 180, 233}
\definecolor{pallette_bluishgreen}{RGB}{0, 158, 115}
\definecolor{pallette_yellow}{RGB}{240, 228, 66}
\definecolor{pallette_blue}{RGB}{0, 114, 178}
\definecolor{pallette_vermillion}{RGB}{213, 94, 0}
\definecolor{pallette_reddishpurple}{RGB}{204, 121, 167}

\AtBeginEnvironment{tikzcd}{\tikzexternaldisable}
\AtEndEnvironment{tikzcd}{\tikzexternalenable}

\newunicodechar{Λ}{\ensuremath{\Lam}}
\newunicodechar{Σ}{\ensuremath{\Sigma}}

\DeclareMathOperator{\Unrav}{\mathrm{Unrav}}

\DeclareMathOperator{\MEF}{\mathrm{MEF}}
\DeclareMathOperator{\mqrank}{\mathrm{mqr}}

\DeclareMathOperator{\Rename}{\mathrm{Rename}}
\DeclareMathOperator{\ModEm}{\mathrm{ModEm}}
\DeclareMathOperator{\Mode}{\ModEm}

\DeclareMathOperator{\Flat}{\mathrm{Coalesce}}
\DeclareMathOperator{\Skel}{\mathrm{Skel}}

\newcommand{\ZFC}{\ensuremath{\mathrm{ZFC}}}

\newcommand{\concat}{}

\DeclareMathOperator{\Rank}{\mathrm{Rank}}

\DeclareMathOperator{\Mod}{\mathrm{Mod}}
\DeclareMathOperator{\Th}{\mathrm{Th}}
\DeclareMathOperator{\Diag}{\mathrm{Diag}}

\DeclareMathOperator{\Tp}{\mathrm{Tp}}

\newcommand{\Q}{\ensuremath{\mathbb{Q}}}

\makeatletter
\newcommand{\@usstar}[1]{{\ua}\left(#1\right)}
\newcommand{\@usnostar}[1]{{\ua}(#1)}
\newcommand{\us}{\@ifstar{\@usstar}{\@usnostar}}
\newcommand{\@dsstar}[1]{{\da}\left(#1\right)}
\newcommand{\@dsnostar}[1]{{\da}(#1)}
\newcommand{\ds}{\@ifstar{\@dsstar}{\@dsnostar}}
\newcommand{\@udsstar}[1]{{\uda}\left(#1\right)}
\newcommand{\@udsnostar}[1]{{\uda}(#1)}
\newcommand{\uds}{\@ifstar{\@udsstar}{\@udsnostar}}
\newcommand{\@Usstar}[1]{{\Ua}\left(#1\right)}
\newcommand{\@Usnostar}[1]{{\Ua}(#1)}
\newcommand{\Us}{\@ifstar{\@Usstar}{\@Usnostar}}
\newcommand{\@Dsstar}[1]{{\Da}\left(#1\right)}
\newcommand{\@Dsnostar}[1]{{\Da}(#1)}
\newcommand{\Ds}{\@ifstar{\@Dsstar}{\@Dsnostar}}
\newcommand{\@Udsstar}[1]{{\Uda}\left(#1\right)}
\newcommand{\@Udsnostar}[1]{{\Uda}(#1)}
\newcommand{\Uds}{\@ifstar{\@Udsstar}{\@Udsnostar}}
\makeatother

\newcommand{\ep}{\epsilon}
\newcommand{\lam}{\lambda}
\newcommand{\Lam}{\ensuremath\Lambda}

\newcommand{\vD}{\vDash}

\newcommand{\bx}{\square}

\newcommand{\ra}{\rightarrow}

\newcommand{\Ra}{\Rightarrow}
\newcommand{\Lra}{\Leftrightarrow}

\newcommand{\ua}{\uparrow}
\newcommand{\da}{\downarrow}
\newcommand{\uda}{\updownarrow}
\newcommand{\Ua}{\Uparrow}
\newcommand{\Da}{\Downarrow}
\newcommand{\Uda}{\Updownarrow}

\newcommand{\sse}{\subseteq}

\newcommand{\ol}{\overline}

\newcommand{\mc}{\mathcal}

\newcommand{\defeq}{\vcentcolon=}

\DeclarePairedDelimiter{\abs}{|}{|}

\newcommand{\bis}{\mathrel{\,\raisebox{.3ex}{$\underline{\makebox[.9em]{$\leftrightarrow$}}$}\,}}

\renewcommand{\leq}{\leqslant}
\renewcommand{\geq}{\geqslant}

\renewcommand{\preceq}{\preccurlyeq}

\renewcommand{\propto}{\varpropto}

\DeclareFontFamily{U} {MnSymbolC}{}

\DeclareFontShape{U}{MnSymbolC}{m}{n}{
	<-6> MnSymbolC5
	<6-7> MnSymbolC6
	<7-8> MnSymbolC7
	<8-9> MnSymbolC8
	<9-10> MnSymbolC9
	<10-12> MnSymbolC10
	<12-> MnSymbolC12}{}
\DeclareFontShape{U}{MnSymbolC}{b}{n}{
	<-6> MnSymbolC-Bold5
	<6-7> MnSymbolC-Bold6
	<7-8> MnSymbolC-Bold7
	<8-9> MnSymbolC-Bold8
	<9-10> MnSymbolC-Bold9
	<10-12> MnSymbolC-Bold10
	<12-> MnSymbolC-Bold12}{}

\DeclareSymbolFont{MnSyC} {U} {MnSymbolC}{m}{n}
\DeclareMathSymbol{\meddiamond}{\mathbin}{MnSyC}{110}

\newcommand{\raisedia}[2]{\raisebox{0.8\depth}{$#1\meddiamond$}}
\newcommand{\diaraw}{\mathpalette\raisedia\relax}
\newcommand{\dia}{{\diaraw}}

\title{Bisimulations of potentialist systems}
\author[*]{Sam Adam-Day}
\affil[*]{Mathematical Institute, University of Oxford, Andrew Wiles Building, Radcliffe Observatory Quarter, Woodstock Road, Oxford, OX2 6GG, United Kingdom; \href{mailto:adamday@maths.ox.ac.uk}{\nolinkurl{adamday@maths.ox.ac.uk}}}
\date{}

\begin{document}

	\maketitle

	\abstract{
		A potentialist system is a first-order Kripke model based on embeddings. I define the notion of bisimulation for these systems, and provide a number of examples. Given a first-order theory $T$, the system $\Mod(T)$ consists of all models of $T$. We can then take either all embeddings, or all substructure inclusions, between these models. I show that these two ways of defining $\Mod(T)$ are bitotally bisimilar. Next, I relate the notion of bisimulation to a generalisation of the Ehrenfeucht-Fraïsé game, and use this to show the equivalence of the existence of a bisimulation with elementary equivalence with respect to an infinitary language. Finally, I consider the question of when a potentialist system is bitotally bisimilar to a system containing set-many models, providing too different sufficient conditions.
	}

	\renewcommand{\thefootnote}{}
	\renewcommand{\thefootnote}{\arabic{footnote}}
	\addtocounter{footnote}{-2}


\section{Introduction}
\label{sec:intro}

The two main objects of study in this article are potentialist systems and bisimulations. A potentialist system is a collection of first-order structures and embeddings between them in which one can interpret first-order modal logic (for definitions, see \cref{sec:potentialist systems}). Potentialist systems are closely related to first-order Kripke models. The former differ from the latter in the following three ways.
\begin{enumerate*}[label=(\arabic*)]
    \item All instances of the accessibility relation between structures are required to be embeddings; in other words, the accessibility relation is required to be inflationary on the domains.
    \item We allow multiple embeddings between two structures.
    \item We allow class-many structures.
\end{enumerate*}

Potentialist systems provide a way of investigating a mathematical structure within the context of other related structures, making precise the notions of necessity and possibility. A formal treatment of these systems is given in \cite{HamkinsLinnebo:Modal-logic-of-set-theoretic-potentialism}. The study of potentialist systems is called `modal model theory'. This name was introduced in \cite{hamkins2020modal}, which conducts a thorough study of the system $\Mod(T)$ of models of a first-order theory $T$. Potentialist systems of set-theoretic forcing extensions are studied in \cite{hamkins_2003, HamkinsLoewe2008:TheModalLogicOfForcing,HamkinsLoewe2013:MovingUpAndDownInTheGenericMultiverse,HamkinsLinnebo:Modal-logic-of-set-theoretic-potentialism,hamkins2015a}. Further systems of models of arithmetic and set theory are studied in \cite{hamkins2018universal,hamkins2018modal,hamkins_williams_2021}. The relationship between model theory and modal logic is also explored in \cite{saveliev-shapirovsky-2018}.

A bisimulation between two potentialist systems is a system of finite partial isomorphisms between structures in one and structures in the other, such that any increase in the size of the isomorphism or travel along the accessibility relation on one side can be mirrored on the other. A bisimulation is bitotal if every element of every model in both systems takes part in the relation. This notion of bisimulation simultaneously generalises the notion of a bisimulation between propositional Kripke models and the notion of a back-and-forth system for first-order models. One of the key properties which bisimulations satisfy is invariance: if two systems are bisimilar, they satisfy the same modal sentences, even allowing for infinite disjunctions and conjunctions. 

An important class of examples of potentialist systems are those of the form $\Mod(T)$, consisting of every model of some first-order theory $T$. There are two ways of defining $\Mod(T)$: by including all embeddings between models, or just the substructure inclusions. As discussed in \cite{hamkins2020modal}, the latter aligns well with a potentialist philosophy on the theory $T$, in which we view the `universe' as unfolding in stages, with individuals becoming actual and persisting identically through subsequent stages. The former on the other hand typically enjoys better algebraic properties, like directedness, amalgamation or convergence (see \cite[Section~5]{hamkins2020modal} for definitions). In Theorem~29 of that paper, Hamkins and Wołoszyn show that the two resulting systems satisfy the same modal sentences. I go one step further and show that the two systems are in fact bitotally bisimilar in a particularly strong way (\cref{thm:Mod T Mode T bis}). This provides a deeper explanation for why the two are equivalent on modal sentences, and immediately extends the result to equivalence with respect to an infinitary language. Moreover, the bisimulation should allow us to transfer results from the more mathematically amenable system based on embeddings, to the more philosophically relevant one based on inclusions.

Bisimulations have a natural interpretation in terms of infinite games. I elaborate on this connection, defining the notion of a \emph{modal Ehrenfeucht-Fraïssé game}. This furnishes us with the converse of invariance: if two systems agree on all infinitary modal sentences then they are bisimilar.

For the remainder of the article, I consider under what circumstances a class-sized potentialist system is bitotally bisimilar with a set-sized one. First, I restrict attention to the systems $\Mod(T)$ for some first-order theory $T$. I show that if $T$ has infinite models, and all its models of size $\kappa$ are $\aleph_0$-saturated, then $\Mod(T)$ is bitotally bisimilar to a system consisting of set-many models. In the last part I take a different tack, showing that if a potentialist system consists of models of an $\aleph_0$-categorical theory which admits quantifier elimination then it is bitotally bisimilar to a particularly simple system containing set-many models. This result makes no assumption on the modal structure of the system, and as such it might be expected that it wouldn't admit much strengthening. As evidence for this, I provide a partial converse, and give an example where weakening the assumption `admits quantifier elimination' leads to a system which is not bisimilar with a set-sized one.

\section{Potentialist systems}
\label{sec:potentialist systems}

A propositional Kripke model consists of a set of `worlds', a relation on this set, and an assignment to each world of the truth values of each propositional variable \cite{BlackburnPatrick2001Ml,ChagrovAlexander1997Ml}. These models are used to provide a semantics for propositional modal logic. They generalise to first-order Kripke models, in which each world is now an entire first-order structure, and we specify how the elements of connected worlds are related. See \cite{HughesG.E1996Anit,FittingMelvin1998Fml,BenthemJohanvan2010Mlfo} for more on first-order Kripke models.

Potentialist systems specialise first-order Kripke models, in that we require that all relations between worlds be embeddings. However, they also generalise, in that we allow multiple relations between any two worlds, and that the collection of worlds be class-sized. The terminology `potentialist system' was first used in \cite{HamkinsLinnebo:Modal-logic-of-set-theoretic-potentialism}.

\begin{definition}\label{def:potentialist system}
	A \emph{potentialist system} is a collection of first-order structures together with a collection of embeddings between them, containing the identity embeddings and all compositions of embeddings.
\end{definition}

I will refer to the elements of a potentialist system as `worlds', `structures' or `models' (the latter being reserved for the context in which we are considering models of a certain first-order theory). Note that a potentialist system is exactly a category of first-order structures with embeddings. A potentialist system is an instance of the more general notion of a `Kripke category', in which one considers the modal logic of an arbitrary concrete category. Kripke categories are defined and investigated in upcoming work by Wojciech Aleksander Wołoszyn \cite{Wojciech-Kripke-cats}. 

Notice that we allow multiple embeddings between worlds. Sometimes, it is natural to consider potentialist systems in which embeddings are always unique, so the abstracted relationship between the worlds is a partial order.

\begin{definition}
	A potentialist system is \emph{thin} if from any first-order structures in it to any other there is at most one embedding.
\end{definition}

This corresponds with the notion of `thinness' for categories. Later, I will show that restricting to thin potentialist systems does not lose generality, in the sense that every potentialist system is bisimilar with a thin one.

\begin{definition}
	A \emph{pointed system} is a triple $(\mc M, M, \ol a)$, where $\mc M$ is a potentialist system, $M \in \mc M$ and $\ol a$ is a finite tuple of elements of $M$. The \emph{parameter size} of the pointed system is the size of $\ol a$. When the parameter size is $0$, write $(\mc M, M)$.
\end{definition}

Let us now move on to consider the interpretation of formulas in potentialist systems. We are interested in several languages. We start with some first-order, finitary, non-modal language $L$ (specified by its signature of non-logical symbols). The first language to meet is the standard first-order modal language based on $L$.

\begin{definition}
	The language $L^\dia$ is formed by adding $\dia$ and $\bx$ as operators. In other words, $L^\dia$ is closed under boolean combinations, quantification and the follow two new rules: whenever $\phi$ is a formula, so are $\dia\phi$ and $\bx\phi$.
\end{definition}

Now, when we come to consider the relationship between bisimulations and sentence satisfaction, we will need to talk about infinitary languages. The definition generalises that of infinitary first-order (non-modal) languages \cite{HodgesWilfrid1993Mt}.

\begin{definition}
	For $\kappa$ an infinite cardinal, let $L_\kappa$ be the first-order language in which we allow conjunction and disjunction of size less than $\kappa$, with the proviso that all formulas contain only finitely many free variables. Let $L_\infty$ be the (class-sized) union of each $L_\kappa$. The languages $L^\dia_\kappa$ and $L^\dia_\infty$ are defined similarly, this time including $\dia$ and $\bx$ as modal operators.
\end{definition}

\begin{remark}
	In non-modal first-order infinitary languages, one can also consider allowing infinitely long quantifier blocks. With this in mind the language $L_\kappa$ is elsewhere commonly denoted by $L_{\kappa, \omega}$, where the `$\omega$' subscript makes explicit the fact that quantifier blocks are all finite; likewise $L_\infty$ is written $L_{\infty, \omega}$. Similarly, in the modal case, one might consider allowing infinitely long blocks of modal operators. So, strictly speaking, we could write $L^\dia_{\kappa,\omega,\omega}$ for $L^\dia_{\kappa}$ and $L^\dia_{\infty,\omega,\omega}$ for $L^\dia_{\infty}$. However, since infinitely long quantifier or modal operator blocks are not considered here, I suppress the additional subscripts.
\end{remark}

Finally, since formula trees are well-founded, one can assign a rank to each formula in $L^\dia_\infty$. This provides a useful complexity stratification, which will be relevant when we come to consider modal Ehrenfeucht-Fraïssé games.


\begin{definition}
	The \emph{modal-quantifier rank}, $\mqrank(\phi)$, of a formula $\phi \in L^\dia_\infty$ is an ordinal defined recursively on the construction of $\phi$.
	\begin{align*}
		\mqrank(\psi)&=0 \tag{when $\psi$ is atomic} \\
		\mqrank(\neg \phi) &= \mqrank(\phi) \\
		\mqrank\left(\bigvee_{i \in I} \phi_i\right) = \mqrank\left(\bigwedge_{i \in I} \phi_i\right) &= \sup_{i \in I} \mqrank(\phi_i) \\
		\mqrank(\exists x \phi) = \mqrank(\forall x \phi) &= \mqrank(\phi)+1 \\
		\mqrank(\dia \phi) = \mqrank(\bx \phi) &= \mqrank(\phi)+1
	\end{align*}
	Write $L^{\dia,\alpha}_\infty$ for the class of $L^\dia_\infty$-formulas of modal-quantifier rank less than $\alpha$.
\end{definition}

The interpretation of first-order modal formulas extends that of classical formulas by interpreting $\dia \phi$ as `$\phi$ holds in some extension' and $\bx \phi$ as the dual: `$\phi$ holds in every extension'.

\begin{definition}
	Let $(\mc M, M, \ol a)$ be a pointed system, and $\phi$ be an $L^\dia_\infty$-formula. We define what it means for $\phi$ to be true at this pointed system, which we write as $\mc M, M \vD \phi[\ol a]$, by induction on the complexity of $\phi$. The non-modal cases are as usual. We let $\mc M, M \vD \dia \phi[\ol a]$ if and only if there is $\pi \colon M \to M'$ in $\mc M$ such that $\mc M, M' \vD \phi[\pi(\ol a)]$. We let $\mc M, M \vD \bx \phi[\ol a]$ if and only if for every $\pi \colon M \to M'$ in $\mc M$ we have $\mc M, M' \vD \phi[\pi(\ol a)]$.
\end{definition}

Note that $\bx \phi$ is equivalent to $\neg \dia \neg \phi$, so in our inductions on formulas we won't need to treat the $\bx \phi$ case separately.

\begin{remark}
	There are some definability issues here, stemming from the fact that the collection of worlds in $\mc M$ may be a proper class. For instance, if we take $\mc M$ to be the class of all initial segments $V_\alpha$ of the cumulative hierarchy under subset inclusion, then $L^\dia$ in $\mc M$ can capture first-order truth in the ambient universe (see \cite[Theorem~1]{HamkinsLinnebo:Modal-logic-of-set-theoretic-potentialism} and \cite{linnebo_2013}). This issue is discussed in Section~9 of \cite{hamkins2020modal}. Following the discussion there, to make sense of the previous definition we can view it as taking place within Morse-Kelley set theory, or Von Neumann-Bernays-Gödel set theory with the axiom of elementary transfinite recursion. Alternatively, we can choose to work with a sequence of Grothendieck-Zermelo universes, so that the extension of some potentialist system in one universe is a set in the next universe up. Note that \ZFC-definability issues don't arise for all class-sized systems. Indeed, Sections~\ref{sec:mod T set} and \ref{sec:omega categorical} provide examples of class-sized systems which are strongly equivalent to set-sized ones.
\end{remark}

\begin{definition}
	The relations $\equiv$, $\equiv^\dia$, $\equiv^\dia_\infty$ and $\equiv^{\dia,\alpha}_\infty$ between pointed systems are the elementary equivalence relations with respect to $L$, $L^\dia$, $L^\dia_\infty$ and $L^{\dia,\alpha}_\infty$ respectively, taking the tuples in the pointed systems as parameters. For example, $(\mc M, M, \ol a) \equiv^\dia_\infty (\mc N, N, \ol b)$ means that for every $L^\dia_\infty$-formula $\phi$, we have $\mc M, M \vD \phi[\ol a]$ if and only if $\mc N, N \vD \phi[\ol b]$.
\end{definition}

Let us meet an important class of potentialist systems, first defined in \cite{hamkins2020modal}. Let $T$ be a first-order $L$-theory. The potentialist system $\Mod(T)$ will be the system based on the collection of all models of $T$. There are two ways of defining this, depending on whether we take the relations between structures to be inclusions or embeddings.

\begin{definition}
	The \emph{substructure potentialist system}, $\Mod(T)$, based on $T$, consists of the class of models of $T$ together with all substructure inclusions. That is, we include all and only the substructure inclusion maps $\iota \colon M \sse N$.
\end{definition}

Of course, $\Mod(T)$ is a thin potentialist system. Note also that the structure of $\Mod(T)$ depends on the underlying domains of the models, in the sense that $M$ may be a substructure of $N$ but not of an isomorphic copy of $N$. It is natural then to consider allowing all embeddings between models, so that the relations between models depend only on their isomorphism types.

\begin{definition}
	The \emph{embedding potentialist system}, $\Mode(T)$, based on $T$ consists of the class of models of $T$ together with all embeddings. In other words, from $M$ to $N$ we include all maps $M \to N$ which are embeddings of $L$-structures.
\end{definition}

In \cref{thm:Mod T Mode T bis} below, I will show that $\Mod(T)$ and $\Mode(T)$ are bitotally bisimilar in a strong sense. As noted in the introduction, this provides and important and deep connection between the philosophically interesting $\Mod(T)$ on the one hand, and the mathematically tractable $\Mode(T)$ on the other.

\section{Bisimulations}
\label{sec:bisimulations}

The notion of bisimulation is central in propositional modal logic (see for example \cite[\S2.2]{BlackburnPatrick2001Ml}). It can be extended to first-order modal logic, at the same time generalising the notion of a back-and-forth system for first-order structures. The following is adapted from Definition 11.4.2 in \cite[p.~123]{BenthemJohanvan2010Mlfo}.

\begin{definition}
	A \emph{bisimulation} between potentialist systems $\mc M$ and $\mc N$ in the same language $L$ is a non-empty relation $\sim$ matching pairs $(M,\ol a)$ with pairs $(N, \ol b)$, where $(\mc M, M,\ol a)$ and $(\mc N, N, \ol b)$ are pointed systems of the same parameter size, such that when $(M,\ol a) \sim (N, \ol b)$ the following hold.
	\begin{enumerate}[label=(B\arabic*), leftmargin=*, labelindent=1pt]
		\item \label{item:partial isomorphism; def:bisimulation} 
			The map $\ol a \mapsto \ol b$ is a partial isomorphism between $M$ and $N$.
		\item \label{item:isomorphism extension; def:bisimulation} 
			For any $c \in M$ there is $d \in N$ such that $(M, \ol a \concat c) \sim (N, \ol b \concat d)$, and vice versa (where $\ol a \concat c$ is the result of adding $c$ at then end of the tuple $\ol a$, etc.).
		\item \label{item:back and forth; def:bisimulation} 
			For any $\pi \colon M \to M'$ in $\mc M$ there is $\rho \colon N \to N'$ in $\mc N$ such that: 
		\begin{equation*}
			(M', \pi(\ol a)) \sim (N', \rho(\ol b))
		\end{equation*}
		and vice versa.
	\end{enumerate}
	Write $(\mc M, M, \ol a) \bis (\mc N, N, \ol b)$ when there is a bisimulation $\sim$ between $\mc M$ and $\mc N$ such that:
	\begin{equation*}
		(M, \ol a) \sim (N, \ol b)
	\end{equation*}
\end{definition}

\begin{definition}
	A bisimulation $\sim$ is \emph{left total} if every pair $(M, \ol a)$ on the left is related to some pair $(N, \ol b)$ on the right. It is \emph{right total} if every pair $(N, \ol b)$ on the right is related to some pair $(M, \ol a)$ on the left. It is \emph{bitotal} if it is both left and right total.
\end{definition}

A bisimulation is a kind of back-and-forth system corresponding to a modal version of the Ehrenfeucht-Fraïssé game. This game-theoretic side will be examined in more detail in \cref{sec:games}. There, it will also be shown that bisimulations correspond to $L^\dia_\infty$-elementary equivalence. We can get a more fine-grained notion corresponding to $L^{\dia,\alpha+1}_\infty$-elementary equivalence as follows.

\begin{definition}\label{def:alpha bisimulation}
	Let $\alpha$ be an ordinal. An \emph{$\alpha$-bisimulation} between potentialist systems $\mc M$ and $\mc N$ in the same language $L$ is a collection of non-empty relations $\sim_\beta$ for $\beta \leq \alpha$. Each relates pairs $(M,\ol a)$ with pairs $(N, \ol b)$, where $(\mc M, M,\ol a)$ and $(\mc N, N, \ol b)$ are pointed systems of the same parameter size. Whenever $(M,\ol a) \sim_\beta (N, \ol b)$ we require that the following hold.
	\begin{enumerate}[label=($\alpha$B\arabic*), leftmargin=*, labelindent=3pt]
		\item \label{item:partial isomorphism; def:alpha-bisimulation} 
			The map $\ol a \mapsto \ol b$ is a partial isomorphism between $M$ and $N$.
		\item \label{item:isomorphism extension; def:alpha-bisimulation} 
			For every $\gamma < \beta$: for any $c \in M$ there is $d \in N$ such that we have $( M, \ol a \concat c) \sim_\gamma (N, \ol b \concat d)$, and vice versa.
		\item \label{item:back and forth; def:alpha-bisimulation} 
			For every $\gamma < \beta$: for any $\pi \colon M \to M'$ in $\mc M$ there is $\rho \colon N \to N'$ in $\mc N$ such that $(M', \pi(\ol a)) \sim_\gamma (N', \rho(\ol b))$, and vice versa.
	\end{enumerate}
	Write $(\mc M, M,\ol a) \bis_\alpha (\mc N, N, \ol b)$ to express the existence of an $\alpha$-bisimulation between $\mc M$ and $\mc N$ such that $(M,\ol a) \sim_\alpha (N, \ol b)$.
\end{definition}

\begin{remark}
	The reader might be wondering why an $\alpha$-bisimulation is defined as a collection of relations $\sim_\beta$ for $\beta \leq \alpha$ rather than for $\beta < \alpha$. It would appear that we lack a convenient way of expressing that `$(\mc M, M,\ol a)$ and $(\mc N, N, \ol b)$ are $\alpha$-bisimilar for every $\alpha$ less than some limit ordinal $\delta$'. However, it is not hard to see that this notion just described is equivalent to $(\mc M, M,\ol a)$ and $(\mc N, N, \ol b)$ being $\delta$-bisimilar. Moreover, it is desirable that being $0$-bisimilar should coincide with being equivalent on atomic formulas.
\end{remark}

\begin{remark}
	In \cref{def:alpha bisimulation} there is no requirement made that the collection of relations $\sim_\beta$ cohere. But note that without loss of generality we can always assume that $\sim_\beta \sse \sim_\gamma$ for $\gamma < \beta$.
\end{remark}

Now, one way of thinking about a bisimulation is as a collection of partial isomorphisms of finite parts of first-order structures which can extended indefinitely. A special case of this is when these partial isomorphisms uniquely combine to a system of total isomorphisms between worlds in one potentialist system and worlds in the other.

\begin{definition}
	An \emph{iso-bisimulation} between potentialist systems $\mc M$ and $\mc N$ is a non-empty relation $\approx$ between $\mc M$ and $\mc N$ together with a system of isomorphisms $(\xi_{M,N} \colon M \to N \mid M \approx N)$ subject to the following condition. Whenever $\pi \colon M \to M'$ in $\mc M$ and $M \approx N$, there is $\rho \colon N \to N'$ in $\mc N$ such that $M' \approx N'$ and the following diagram commutes; and vice versa.
	\begin{equation*}
		\begin{tikzcd}
			M' \arrow[r, "{\xi_{M',N'}}"]                              & N'                                 \\
			M \arrow[u, "{\pi}"] \arrow[r, "{\xi_{M,N}}"] & N \arrow[u, "{\rho}"']
		\end{tikzcd}
	\end{equation*}
\end{definition}

The following allows us to convert any iso-bisimulation into an (ordinary) bisimulation.

\begin{lemma}
	Let $(\approx,(\xi))$ be an iso-bisimulation between $\mc M$ and $\mc N$. Define the relation $\approx^*$ by setting $(M, \ol a) \approx^* (N, \ol b)$ if and only if $M \approx N$ and $\xi_{M,N}(\ol a) = \ol b$. Then $\approx^*$ is a bisimulation.
\end{lemma}

Hence, any iso-bisimulation gives rise to a bisimulation. The converse is not true in general, e.g.\@ for cardinality reasons, or because a bisimulation can involve multiple non-compatible partial isomorphisms between structures. The following is a simple example of two bisimilar systems which are not iso-bisimilar.

\begin{example}
	Let $L$ be the empty language. Let $\mc M$ be the potentialist system consisting of a single $\aleph_0$-size $L$-structure, and let $\mc N$ be the structure consisting of an $\aleph_0$-sized and an $\aleph_1$-sized $L$-structure, with an embedding from the former to the latter. There is no iso-bisimulation $\mc M \bis \mc N$ for cardinality reasons, but there is an (ordinary) bisimulation $\mc M \bis \mc N$. Indeed, relate pairs according to their atomic types:\footnote{I use the notation $\Delta_0^M(\ol a)$ to denote the atomic type of $\ol a$ in $M$: the set of atomic formulas true in $M$ of $\ol a$.}
	\begin{equation*}
		(M, \ol a) \sim (N, \ol b) \quad\Lra\quad \Delta_0^M(\ol a) = \Delta_0^N(\ol b)
	\end{equation*}
\end{example}

A key property of bisimulations is their invariance: bisimilar systems satisfy the same formulas (c.f.\@ \cite[Theorem~21, p.~124]{BenthemJohanvan2010Mlfo}).

\begin{theorem}[Fundamental Theorem of Bisimulations]\label{thm:bisimulation invariance}
	If two pointed systems are bisimilar then they are $L^\dia_\infty$-elementarily equivalent.
\end{theorem}

\begin{proof}
	The proof is by induction on formulas. Let $\sim$ be a bisimulation between $(\mc M, M, \ol a)$ and $(\mc N, N, \ol b)$. 
	The atomic case follows from \ref{item:partial isomorphism; def:bisimulation}. 
	If we have $\mc M, M \vD \exists x \psi([\ol a], x)$ then there is $c \in M$ such that $\mc M, M \vD \psi[\ol a,c]$; by \ref{item:isomorphism extension; def:bisimulation} there is $d \in N$ such that $(M, \ol a \concat c) \sim (N, \ol b \concat d)$, so by induction hypothesis $\mc N, N \vD \psi[\ol b, d]$ and $\mc N, N \vD \exists x \psi([\ol b], x)$. 
	If $\mc M, M \vD \dia \psi[\ol a]$ then for some $\pi \colon M \to M'$ in $\mc M$ we have $\mc M, M' \vD \psi[\pi(\ol a)]$. By \ref{item:back and forth; def:bisimulation} there is $\rho \colon N \to N'$ in $\mc N$ with $(M', \pi(\ol a)) \sim (N',\rho(\ol b))$; so by induction hypothesis $\mc N, N' \vD \psi[\rho(\ol b)]$ and $\mc N, N \vD \dia \psi[\ol b]$.
\end{proof}

\begin{corollary}\label{cor:iso-bisimulation invariance}
	If $(\approx,(\xi))$ is an iso-bisimulation between $\mc M$ and $\mc N$, then whenever $M \approx N$, for any $\ol a$ in $M$ and formula $\phi$ of $L^\dia_\infty$:
	\begin{equation*}
		\mc M, M \vD \phi[\ol a] \Lra \mc N, N \vD \phi[\xi_{M,N}(\ol a)]
	\end{equation*}
\end{corollary}

\begin{theorem}[Fundamental Theorem of $\alpha$-bisimulations]\label{thm:alpha-bisimulation invariance}
	If two pointed systems are $\alpha$-bisimilar then they are $L^{\dia, \alpha+1}_\infty$-elementarily equivalent.
\end{theorem}

\begin{proof}
	This is a refinement of the proof of \cref{thm:bisimulation invariance}. We prove by induction on modal rank that whenever $(M, \ol a) \sim_\beta (N, \ol b)$, for any $\phi$ with $\mqrank(\phi) \leq \beta$ we have:
	\begin{equation*}
		\mc M, M \vD \phi[\ol a] \quad\Lra\quad \mc N, N \vD \phi[\ol b]
	\end{equation*}
	As above the base case is dealt with by \ref{item:partial isomorphism; def:alpha-bisimulation}. If $M \vD \exists x \psi([\ol a], x)$ there is $c \in M$ such that $M \vD \psi[\ol a, c]$; by \ref{item:isomorphism extension; def:alpha-bisimulation} there is $d \in N$ such that:
	\begin{equation*}
		(M, \ol a \concat c) \sim_{\mqrank(\psi)} (N, \ol b \concat d)
	\end{equation*}
	which by induction hypothesis means that $N \vD \psi[\ol b, d]$. The other cases are similar.
\end{proof}

The converses of \cref{thm:bisimulation invariance} and \cref{thm:alpha-bisimulation invariance} will be proved in \cref{sec:games} after we have established the connection with infinite games.

The following provide promised proofs of facts mentioned in \cref{sec:potentialist systems}. We first see that any potentialist system is iso-bisimilar with a thin one, using the `unravelling' construction, a generalisation from propositional modal logic (see \cite[Definition~4.51]{BlackburnPatrick2001Ml}).

\begin{definition}
	Let $\mc M$ be a potentialist system. The \emph{unravelling}, $\Unrav(\mc M)$, of $\mc M$ is the free category on $\mc M$, regarded as a potentialist system. Concretely, $\Unrav(\mc M)$ is specified as follows.
	\begin{itemize}
		\item The objects are finite sequences of embeddings:
		\begin{equation*}
			\vec\pi = \left(M_0 \xrightarrow{\pi_1} M_1 \xrightarrow{\pi_2} \cdots \xrightarrow{\pi_{n-1}} M_{n-1} \xrightarrow{\pi_n} M_n\right)
		\end{equation*}
		where $M_0, \ldots, M_n$ are worlds from $\mc M$ and $\pi_1, \ldots, \pi_n$ are embeddings from $\mc M$. We allow singleton sequences. The structure corresponding to $\vec\pi$ is $M_{\vec\pi} = M_n$.
		\item For every sequence $\vec\pi$ and every embedding $\pi_{n+1}$ with domain $M_{\vec\pi}$, we obtain a new sequence $\vec\pi \pi_{n+1}$ by adding $\pi_{n+1}$ at the end, and we have an arrow $\pi_{n+1} \colon \vec\pi \to \vec\pi \pi_{n+1}$.
	\end{itemize}
\end{definition}

\begin{remark}
	This definition is slightly informal. According to \cref{def:potentialist system}, worlds in a potentialist system should be $L$-structures, but here they are finite sequences. So to formally define the $\Unrav(\mc M)$ we should take the worlds to be the last structures $M_{\vec\pi} = M_n$, and somehow encode the sequence structure in the domain. For example, we could rename the elements of the domain as $M_{\vec\pi} \times \{(\pi_1 \cdots \pi_n)\}$. I won't worry too much about these technicalities from now on.
\end{remark}

Note that the unravelling of $\mc M$ is always a thin category. The following gives the bisimulation $\mc M \bis \Unrav(\mc M)$.

\begin{lemma}\label{res:unravelling bisimulation}
	Let $\mc M$ be a potentialist system. The relation $\approx$ between $\mc M$ and $\Unrav(\mc M)$, given by: $N \approx \vec\pi$ if and only if $M_{\vec\pi}=N$, together with the identity isomorphisms, is an iso-bisimulation which is bitotal.
\end{lemma}

\begin{proof}
	Assume that $N \approx \vec\pi$. Take $\rho \colon N \to N'$ in $\mc M$. Then $\vec\pi\rho \in \Unrav(\mc M)$ and $N' \approx \vec\pi\rho$. Furthermore, we have: 
	\begin{equation*}
		\rho \colon \vec\pi \to \vec\pi\rho
	\end{equation*}
	in $\Unrav(\mc M)$. Conversely, if $\rho \colon \vec\pi \to \vec\pi\rho$ in $\Unrav(\mc M)$ with codomain $N'$, then we have $\rho \colon N \to N'$ in $\mc M$.
\end{proof}

\begin{corollary}
	Every potentialist system is bisimilar to thin one via a bitotal iso-bisimulation.
\end{corollary}

Let us see now that $\Mod(T)$ and $\Mode(T)$ are bisimilar. 

\begin{theorem}\label{thm:Mod T Mode T bis}
	Let $T$ be a first-order $L$-theory. There is a bitotal iso-bisimulation between the systems $\Mod(T)$ and $\Mode(T)$.
\end{theorem}

We first need to define a basic disjointifying bisimulation, which makes sure that worlds are disjoint.

\begin{definition}
	Let $\mc M$ be a potentialist system. The \emph{disjointification} of $\mc M$, notation $\mc M^\bullet$, is the potentialist system obtained from $\mc M$ by replacing the domain of each world $M$ with $M \times \{M\}$, then taking all the embeddings from $\mc M$, suitably modified. The \emph{disjointifying bisimulation} $\rightthreetimes$ is the iso-bisimulation $\mc M \bis \mc M^\bullet$ defined by relating each world in $\mc M^\bullet$ with its original copy in $\mc M$:
	\begin{equation*}
		M \rightthreetimes M \times \{M\} 
	\end{equation*}
\end{definition}

Now, the idea to prove \cref{thm:Mod T Mode T bis} is that we first take the unravelling $\Unrav(\Mode(T)^\bullet)$, and then start renaming the elements of its worlds so that embeddings become inclusions. In order to do this consistently, we need to carry out this process iteratively through sequences of embeddings. See \cref{fig:iterative renaming necessary} for an illustration. Fortunately, every world in $\Unrav(\Mode(T)^\bullet)$ remembers a sequence of previous embeddings, and this allows us to carry out the process.

\begin{figure}
	\centering
	\begin{tikzpicture}
		\begin{scope}
			\draw[pattern=crosshatch, pattern color=pallette_orange] (0,0) circle (0.5);
			\draw[pattern=horizontal lines, pattern color=pallette_skyblue] (0,2) ellipse (1 and 0.75);
			\draw[pattern=north east lines, pattern color=pallette_bluishgreen] (0,4.5) ellipse (1.5 and 1);
			\draw[->] (0,0.55) -- (0,1.2);
			\draw[->] (0,2.8) -- (0,3.45);
		\end{scope}
		\begin{scope}[xshift=150]
			\draw[pattern=crosshatch, pattern color=pallette_orange] (0,0) circle (0.5);

			\draw[pattern=horizontal lines, pattern color=pallette_skyblue] (0,2) ellipse (1 and 0.75);
			\fill[white] (0,2) circle (0.5);
			\draw[pattern=crosshatch, pattern color=pallette_orange] (0,2) circle (0.5);

			\draw[pattern=north east lines, pattern color=pallette_bluishgreen] (0,4.5) ellipse (1.5 and 1);
			\fill[white] (0,4.5) ellipse (1 and 0.75);
			\draw[pattern=horizontal lines, pattern color=pallette_skyblue] (0,4.5) ellipse (1 and 0.75);
			\fill[white] (0,4.5) circle (0.5);
			\draw[pattern=crosshatch, pattern color=pallette_orange] (0,4.5) circle (0.5);
			\draw[right hook->] (0,0.55) -- (0,1.2);
			\draw[right hook->] (0,2.8) -- (0,3.45);
		\end{scope}
	\end{tikzpicture}
	\caption{An example showing that iteratively renaming elements is necessary. Blobs represent worlds and arrows embeddings. Hook arrows represent inclusions. To transform the sequence on the left so that embeddings become inclusions, we need to iteratively rename along the embeddings, as on the right.}
	\label{fig:iterative renaming necessary}
\end{figure}
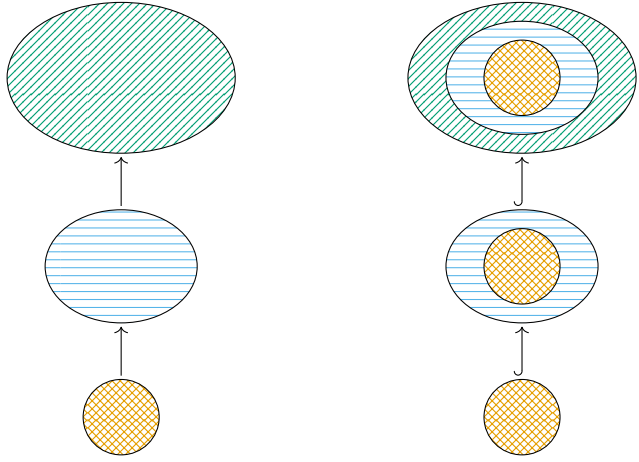

\begin{definition}
	Let $\mc M$ be a potentialist system, and take a sequence $\vec\pi = \left(M_0 \xrightarrow{\pi_1}  \cdots \xrightarrow{\pi_n} M_n\right) \in \Unrav(\mc M^\bullet)$. The structure $\Flat(\vec\pi)$ is the isomorphic copy of $M_n$ defined iteratively by renaming $\pi_1[M_0]$ as $M_0$ in $M_1$, then renaming $\pi_2[M_1]$ using the result of this, and so on, as illustrated in \cref{fig:iterative renaming necessary}.
\end{definition}

\begin{remark}
	The disjointified version $\mc M^\bullet$ of $\mc M$ is necessary because the structure $\Flat(\vec\pi)$ is only sensible if no renamed element becomes the same as a non-renamed element. For instance, in the case $n=1$, if there is $a \in M_0 \cap (M_1 \setminus \pi_1[M_0])$, then the elements  $a, \pi_1(a) \in M_0$ will have collapsed to a single one $a \in \Flat(\vec\pi)$.
\end{remark}

\begin{definition}
	Let $T$ be any first-order $L$-theory. Let the \emph{renaming bisimulation} $\propto$ between $\Mod(T)^\bullet$ and $\Unrav(\Mode(T)^\bullet)$ be the iso-bisimulation defined as follows. For any $\vec\pi \in \Unrav(\Mode(T)^\bullet)$ put:
	\begin{equation*}
		\Flat(\vec\pi) \propto \vec\pi
	\end{equation*}
	The isomorphism $\Rename_{\vec\pi} \colon \Flat(\vec\pi) \to M_{\vec\pi}$ is the natural renaming isomorphism.
\end{definition}

\begin{theorem}\label{thm:renaming bis total bis}
	The renaming bisimulation is a bitotal iso-bisimulation.
\end{theorem}

\begin{proof}
	First note that for any $M_0 \in \Mod(T)^\bullet$, we have a length-$0$ sequence $\ep_{M_0} \in \Unrav(\Mode(T)^\bullet)$ with $M_0$ as the initial world, and $M_0 \propto \ep_{M_0}$; hence $\propto$ is left total. It is clearly right total. So we need to check the isomorphism extension property for iso-bisimulations. Take $\vec\pi \in \Unrav(\Mod(T)^\bullet)$, so that $\Flat(\vec\pi) \propto \vec\pi$.

	Assume that $\Flat(\vec\pi) \sse N$ in $\Mod(T)$. Define $\rho \colon M_{\vec\pi} \to N$ to be the result of renaming in $M_{\vec\pi}$ via $\Rename_{\vec\pi}^{-1}$, then applying the inclusion embedding into $N$. Then in fact:
	\begin{equation*}
		N = \Flat(\vec\pi \rho)
	\end{equation*}
	We have $\rho \colon \vec\pi \to \vec\pi \rho$ in $\Unrav(\Mode(T)^\bullet)$, and the following commutative diagram, where the hook arrow represents an inclusion. See also \cref{fig:defining rho from Pi; proof:renaming bis total bis} for a picture-based commutative diagram.
	\begin{equation*}
		\begin{tikzcd}
			{N = \Flat(\vec\pi \rho)} \arrow[rr, "\Rename_{\vec\pi \rho}"]        &  & N                     \\
			                                                   &  &                        \\
			{\Flat(\vec\pi)} \arrow[rr, "\Rename_{\vec\pi}"] \arrow[uu, hook] &  & M_{\vec\pi} \arrow[uu, "\rho"]
		\end{tikzcd}
	\end{equation*}

	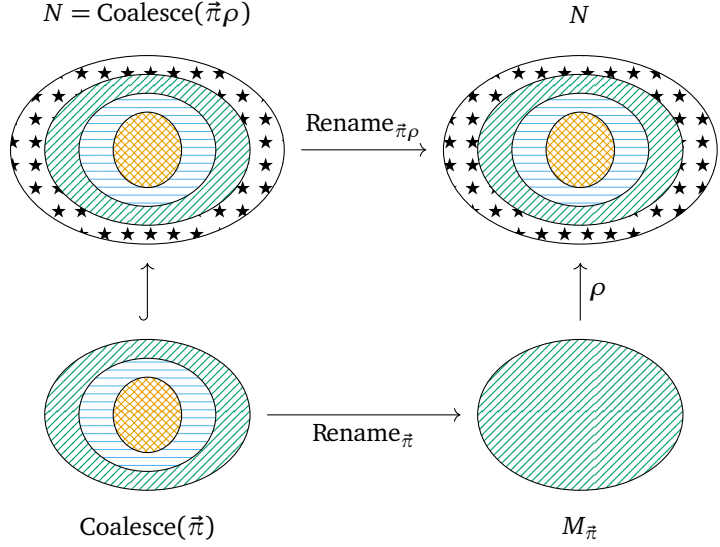
\begin{figure}
		\centering
		\begin{tikzpicture}[xscale=0.9]
			\begin{scope}
				\draw[pattern=north east lines, pattern color=pallette_bluishgreen] (0,0) ellipse (1.5 and 1);
				\fill[white] (0,0) ellipse (1 and 0.75);
				\draw[pattern=horizontal lines, pattern color=pallette_skyblue] (0,0) ellipse (1 and 0.75);
				\fill[white] (0,0) circle (0.5);
				\draw[pattern=crosshatch, pattern color=pallette_orange] (0,0) circle (0.5);
				\coordinate (blt) at (0,1.25);
				\coordinate (blr) at (1.75,0);
				\node at (0,-1.5) {$\Flat(\vec\pi)$};
			\end{scope}
			\begin{scope}[yshift=100]
				\draw[pattern=fivepointed stars, pattern color=black] (0,0) ellipse (2 and 1.25);
				\fill[white] (0,0) ellipse (1.5 and 1);
				\draw[pattern=north east lines, pattern color=pallette_bluishgreen] (0,0) ellipse (1.5 and 1);
				\fill[white] (0,0) ellipse (1 and 0.75);
				\draw[pattern=horizontal lines, pattern color=pallette_skyblue] (0,0) ellipse (1 and 0.75);
				\fill[white] (0,0) circle (0.5);
				\draw[pattern=crosshatch, pattern color=pallette_orange] (0,0) circle (0.5);
				\coordinate (tlb) at (0,-1.5);
				\coordinate (tlr) at (2.25,0);
				\node at (0,1.8) {$N = \Flat(\vec\pi \rho)$};
			\end{scope}
			\begin{scope}[xshift=180]
				\draw[pattern=north east lines, pattern color=pallette_bluishgreen] (0,0) ellipse (1.5 and 1);
				\coordinate (brt) at (0,1.25);
				\coordinate (brl) at (-1.75,0);
				\node at (0,-1.5) {$M_{\vec\pi}$};
			\end{scope}
			\begin{scope}[xshift=180,yshift=100]
				\draw[pattern=fivepointed stars, pattern color=black] (0,0) ellipse (2 and 1.25);
				\fill[white] (0,0) ellipse (1.5 and 1);
				\draw[pattern=north east lines, pattern color=pallette_bluishgreen] (0,0) ellipse (1.5 and 1);
				\fill[white] (0,0) ellipse (1 and 0.75);
				\draw[pattern=horizontal lines, pattern color=pallette_skyblue] (0,0) ellipse (1 and 0.75);
				\fill[white] (0,0) circle (0.5);
				\draw[pattern=crosshatch, pattern color=pallette_orange] (0,0) circle (0.5);
				\coordinate (trb) at (0,-1.5);
				\coordinate (trl) at (-2.25,0);
				\node at (0,1.8) {$N$};
			\end{scope}
			\draw[right hook->] (blt) -- (tlb);
			\draw[->] (blr) -- (brl) node[midway,below] {$\Rename_{\vec\pi}$};
			\draw[->] (tlr) -- (trl) node[midway,above] {$\Rename_{\vec\pi \rho}$};
			\draw[->] (brt) -- (trb) node[midway,right] {$\rho$};
		\end{tikzpicture}
		\caption{Defining $\rho \colon M_{\vec\pi} \to N$ from $\Flat(\vec\pi) \sse N$. The hook arrow represents an inclusion.}
		\label{fig:defining rho from Pi; proof:renaming bis total bis}
	\end{figure}

	Conversely, take $\rho \colon \vec\pi \to \vec\pi \rho$ in $\Unrav(\Mode(T)^\bullet)$ with codomain $N$. Note that by definition $\Flat(\vec\pi) \sse \Flat(\vec\pi \rho)$, and this is a substructure relation. Moreover, the following diagram commutes.
	\begin{equation*}
		\begin{tikzcd}
			{\Flat(\vec\pi \rho)} \arrow[rr, "\Rename_{\vec\pi \rho}"]        &  & N                     \\
			                                                   &  &                        \\
			{\Flat(\vec\pi)} \arrow[rr, "\Rename_{\vec\pi}"] \arrow[uu, hook] &  & M_{\vec\pi} \arrow[uu, "\rho"]
		\end{tikzcd}
	\end{equation*}\qedhere
\end{proof}

\begin{proof}[Proof of \cref{thm:Mod T Mode T bis}]
	Indeed:
	\begin{equation*}
		\Mod(T) \bis \Mod(T)^\bullet \bis \Unrav(\Mode(T)^\bullet) \bis \Mode(T)^\bullet \bis \Mode(T)
	\end{equation*}
	via bitotal iso-bisimulations.
\end{proof}

\section{Modal Ehrenfeucht-Fraïssé games}
\label{sec:games}

We now turn to the connection between bisimulations and games. Modal Ehrenfeucht-Fraïssé games generalise their non-modal cousins (see \cite[\S3.2]{HodgesWilfrid1993Mt}).

\begin{definition}
	Let $(\mc M, M, \ol {a})$ and $(\mc N, N, \ol {b})$ be pointed systems of the same parameter size. The \emph{modal Ehrenfeucht-Fraïssé game}:
	\begin{equation*}
		\MEF(\mc M, M, \ol {a}, \mc N, N, \ol {b})
	\end{equation*}
	is played by Eloise and Abelard. The positions of the game are quadruples $(U, \ol {c}, Q, \ol {d})$, where $(\mc M, U, \ol {c})$ and $(\mc N, V, \ol {d})$ are pointed systems of the same parameter size. The initial position is $( M, \ol {a}, N, \ol {b})$. The game proceeds in moves alternating between Abelard and Eloise. From position $(U, \ol {c}, V, \ol {d})$, Abelard can make one of the following two kinds of moves.
	\begin{enumerate}[label=(\alph*)]
		\item Choose $u \in U$ or $v \in V$.
		\item Choose $\pi \colon U \to U'$ in $\mc M$ or $\rho \colon V \to V'$ in $\mc N$.
	\end{enumerate}
	Eloise then responds, respectively, as follows.
	\begin{enumerate}[label=(\alph*)]
		\item When Abelard chose $u \in U$, choose $v \in V$; otherwise choose $u \in U$. The new position is then $(U, \ol {c} u, V, \ol {d} v)$.
		\item When Abelard chose $\pi \colon U \to U'$, choose $\rho \colon V \to V'$; otherwise choose $\pi \colon U \to U'$. The new position is then $(U', \pi(\ol {c}), V', \rho(\ol {d}))$.
	\end{enumerate}
	A play in the game, consisting of an $\omega$-sequence of positions, is a win for Eloise if and only if for every position $(U, \ol {c}, V, \ol {d})$ in the play, the atomic types agree (in other words $\ol c \mapsto \ol d$ is a partial isomorphism):
	\begin{equation*}
		\Delta_0^U(\ol c) = \Delta_0^V(\ol d)
	\end{equation*}
	Often, we want to consider just the set of all positions in the game based on $\mc M$ and $\mc M'$, together with the winning conditions, without fixing an initial position. I will call this the \emph{board} and denote it by $\MEF(\mc M, \mc M')$. This becomes a game when we specify an initial position.
\end{definition}

The following theorem demonstrates the connection between bisimulations and modal Ehrenfeucht-Fraïssé games. It simultaneously generalises the connection between non-modal Ehrenfeucht-Fraïssé games (based on two fixed structures) and back-and-forth systems, on the one hand \cite[Lemma 3.2.2]{HodgesWilfrid1993Mt}, and bisimulation games and (propositional) bisimulations, on the other \cite[\S15.7]{10.7551/mitpress/9674.001.0001}.

\begin{theorem}\label{thm:Eloise winning iff bis}
	Eloise has a winning strategy in the game $\MEF(\mc M, M, \ol a, \mc N, N, \ol b)$ if and only if $(\mc M, M, \ol a)$ is bisimilar with $(\mc N, N, \ol b)$.
\end{theorem}

\begin{proof}
	Let $\sim$ be a bisimulation $(\mc M, M, \ol a) \bis (\mc N, N, \ol b)$. Eloise uses $\sim$ to see how she should play. Note that the conditions \ref{item:isomorphism extension; def:bisimulation} and \ref{item:back and forth; def:bisimulation} guarantee that no matter what Abelard plays, Eloise can always ensure the play stays at positions $(U, \ol {c},V, \ol {d})$ such that $(\mc M, U, \ol {c}) \sim (\mc N, V, \ol {d})$.\footnote{Note that in general we need global choice in order to define Eloise's strategy from the bisimulation.} Then \ref{item:partial isomorphism; def:bisimulation} guarantees that for such positions:
	\begin{equation*}
		\Delta_0^U(\ol c) = \Delta_0^V(\ol d)
	\end{equation*}
	meaning that the resulting play is a win for Eloise.

	Now assume that Eloise has a winning strategy in $\MEF(\mc M, M, \ol a, \mc N, N, \ol b)$. We can construct a bisimulation by playing as Abelard and seeing what Eloise does. Put $(U, \ol {c}) \sim (V, \ol {d})$ if and only if the position $(U, \ol {c},V, \ol {d})$ occurs in some play in which Eloise plays her winning strategy. Let us see that this is a bisimulation. Condition \ref{item:partial isomorphism; def:bisimulation} follows from the fact that the atomic types of such positions must agree. As for condition \ref{item:isomorphism extension; def:bisimulation}, take $u \in U$. We make Abelard choose this $u \in U$, and then according to her winning strategy Eloise then plays some $v \in V$, with the result that $(U, \ol {c} u) \sim (V, \ol {d} v)$. The converse direction, and condition \ref{item:back and forth; def:bisimulation}, are similar.
\end{proof}

In the first-order case, the existence of a winning strategy for Eloise in the Ehrenfeucht-Fraïssé game can be expressed by a formula in an infinitary language (see \cite[\S3.5]{HodgesWilfrid1993Mt}). An analogous result holds for modal Ehrenfeucht-Fraïssé games, allowing us to show that two pointed systems are bisimilar if and only if they are $L^\dia_\infty$-elementarily equivalent. This further yields a converse to the Fundamental Theorem of Bisimulations (\cref{thm:bisimulation invariance}).

\begin{theorem}\label{thm:MEF tfae}
	The following are equivalent.
	\begin{enumerate}[label=(\arabic*)]
		\item\label{item:winning; thm:MEF tfae}
			Eloise has a winning strategy in $\MEF(\mc M, M, \ol a, \mc N, N, \ol b)$.
		\item\label{item:elem equiv; thm:MEF tfae}
			$(\mc M, M, \ol a)$ is $L^\dia_\infty$-elementarily equivalent to $(\mc N, N, \ol b)$.
		\item\label{item:bisimulation; thm:MEF tfae}
			There is a bisimulation $(\mc M, M, \ol a) \bis (\mc N, N, \ol b)$.
	\end{enumerate}
\end{theorem}

The version of the equivalence $\text{\ref{item:winning; thm:MEF tfae}} \Lra \text{\ref{item:elem equiv; thm:MEF tfae}}$ for non-modal Ehrenfeucht-Fraïssé games is Theorem~3.5.2 in \cite{HodgesWilfrid1993Mt}. I follow the proof structure given there, modifying to account for:
\begin{enumerate*}[label=(\arabic*)]
	\item the additional modal aspect, and
	\item the fact that potentialist systems can be class-sized.
\end{enumerate*}
We first need some terminology concerning the ranks of game positions.

\begin{definition}
	The \emph{rank}, notation $\Rank(p)$, of a position $p=(M, \ol a, N, \ol b)$ in $\MEF(\mc M, \mc N)$ is either $-1$, an ordinal, or $\infty$, and is determined as follows.
	\begin{itemize}
		\item $\Rank(p) \geq 0$ if and only if $\Delta_0^M(\ol a) = \Delta_0^N(\ol {a})$.
		\item $\Rank(p) \geq \alpha + 1$ if and only if for every possible move made by Abelard from $p$, Eloise can then move to a position of rank at least $\alpha$.
		\item $\Rank(p) \geq \lam$, for $\lam$ a limit ordinal, if and only if $\Rank(p) \geq \alpha$ for all $\alpha < \lam$.
	\end{itemize}
\end{definition}

\begin{lemma}\label{lem:alpha equiv iff alpha bis}
	For any ordinal $\alpha$ and pointed systems $(\mc M, M, \ol a)$ and $(\mc N, N, \ol b)$ of the same parameter size we have:
	\begin{equation*}
		\Rank(M, \ol a, N, \ol b) \geq \alpha \text{ in }\MEF(\mc M, \mc N) \quad\Lra\quad (\mc M, M, \ol a) \bis_{\alpha} (\mc N, N, \ol b)
	\end{equation*}
\end{lemma}

\begin{proof}
	This is a refinement of the proof of \cref{thm:Eloise winning iff bis}. Put $(U, \ol {c}) \sim_\beta (V, \ol {d})$ if and only if $\beta \leq \min\{\Rank(U, \ol {c}, V, \ol {d}), \alpha\}$. This then gives an $\alpha$-bisimulation $(\mc M, M, \ol a) \bis_{\alpha} (\mc N, N, \ol b)$. Conversely, if if we have any $\alpha$-bisimulation $(\mc M, M, \ol a) \bis_{\alpha} (\mc N, N, \ol b)$, then any instance $(U, \ol {c}) \sim_\beta (V, \ol {d})$ witnesses that $\beta \leq \Rank(U, \ol {c}, V, \ol {d})$.
\end{proof}

\begin{lemma}\label{lem:position rank basic results}\
	\begin{enumerate}[label=(\arabic*)]
		\item\label{item:alpha to alpha plus; lem:position rank basic results}
			If $\alpha$ is an ordinal such that every position of rank at least $\alpha$ is also at least $\alpha+1$, then every position of rank at least $\alpha$ is winning for Eloise.
		\item\label{item:infty iff winning; lem:position rank basic results}
			A position has rank $\infty$ if and only if it is winning for Eloise. Hence:
			\begin{equation*}
				(\mc M, M, \ol a) \bis (\mc N, N, \ol b) \quad\Lra\quad \forall \alpha \colon (\mc M, M, \ol a) \bis_\alpha (\mc N, N, \ol b)
			\end{equation*}
	\end{enumerate}
\end{lemma}

\begin{proof}
	These facts are proved more generally in Lemma~3.4.1 and Theorem 3.4.2 of \cite{HodgesWilfrid1993Mt}.
\end{proof}

As mentioned above, we can express the property of having rank at least $\alpha$ using an $L^\dia_\infty$-formula of modal-quantifier rank $\alpha$. I will omit the $\mc M$ from the superscript of $\theta$ whenever possible.

\begin{theorem}\label{thm:theta alpha iff rank geq alpha}
	Let $(\mc M, M, \ol a)$ be a pointed system of parameter size $n$. Take an ordinal $\alpha$. There is an $n$-variable formula $\theta^{\mc M, M, \ol a}_{\alpha}$ in $L^{\dia}_\infty$ of modal-quantifier rank exactly $\alpha$ such that for every pointed system $(\mc N, N, \ol b)$ of parameter size $n$, we have:
	\begin{equation*}
		\mc N, N \vD \theta^{\mc M, M, \ol a}_{\alpha}\left[\ol b\right] \quad\Lra\quad \Rank(M, \ol a, N, \ol b) \geq \alpha \text{ in }\MEF(\mc M, \mc N)
	\end{equation*}
\end{theorem}

\begin{proof}
	The formula $\theta^{M, \ol a}_{\alpha}$ is defined by recursion on $\alpha$. 
	\begin{itemize}
		\item 
			First:
			\begin{equation*}
				\theta^{M, \ol a}_0(\ol x) \defeq \bigwedge \Delta_0^M(\ol a)
			\end{equation*}
		\item 
			For successor ordinals $\alpha+1$, we need to make sure that Eloise can always force to a position of rank at least $\alpha$, no matter what Abelard chooses. Considering the definition of the game, this requires expressing the following.
			\begin{enumerate}[label=(\roman*)]
				\item For every $u \in M$ there is $y \in N$ such that $\theta^{M,(\ol a u)}_{\alpha}(\ol x, y)$ holds at $N$.
				\item For every $y \in N$ there is $u \in M$ such that $\theta^{M,(\ol a u)}_{\alpha}(\ol x, y)$ holds at $N$.
				\item For every $\pi \colon M \to M'$ there is $\rho \colon N \to N'$ such that $\theta^{M,\pi(\ol a)}_{\alpha}(\rho(\ol x))$ holds at $N'$.
				\item For every $\rho \colon N \to N'$ there is $\pi \colon M \to M'$ such that $\theta^{M,\pi(\ol a)}_{\alpha}(\rho(\ol x))$ holds at $N'$.
			\end{enumerate}
			The first two conditions can be expressed by replacing quantification over $M$ with conjunctions and disjunctions. For the other two conditions, this strategy does not immediately work, since there may be class-many embeddings $\pi \colon M \to M'$, but we are only allowed to take set-sized conjunctions and disjunctions. Fortunately, it will follow from the induction that there are only set-many formulas of the form $\theta^{M,\pi(\ol a)}_\alpha(\ol x)$.\footnote{In fact, one can show that there are only set-many formulas of rank $\alpha$ up to equivalence.} Let $\Phi$ be the set of formulas $\theta^{M,\pi(\ol a)}_\alpha(\ol x)$ for $\pi \colon M \to M'$ (Even though there may be class-many \emph{names} $\theta^{M,\pi(\ol a)}_\alpha$, there are only set-many formulas named by them.). We can now define:
			\begin{equation*}
				\theta^{M, \ol a}_{\alpha+1}(\ol x) \defeq 
					\left(
					\begin{array}{ll}
						       & \bigwedge_{u \in M} \exists y \theta^{M,(\ol a \concat u)}_{\alpha}(\ol x, y) \\[0.3em]
						\wedge & \forall y \bigvee_{u \in M} \theta^{M,(\ol a \concat u)}_{\alpha}(\ol x, y) \\[0.3em]
						\wedge & \bigwedge_{\phi \in \Phi} \dia \phi \\[0.3em]
						\wedge & \bx \bigvee_{\phi \in \Phi} \phi
					\end{array}
					\right)
			\end{equation*}

		\item 
			Finally, when $\gamma$ is a limit, we define:
			\begin{equation*}
				\theta^{M, \ol a}_{\gamma}(\ol x) \defeq \bigwedge_{\alpha < \gamma} \theta^{M, \ol a}_{\alpha}(\ol x) 
			\end{equation*}
	\end{itemize}
	Note that at each stage the definition of $\theta^{M, \ol a}_{\alpha}$ corresponds with the condition that $\Rank(M, \ol a, N, \ol b) \geq \alpha$.
\end{proof}

With this key piece we can establish a correspondence between all of the notions we have of `equivalence up to rank $\alpha$'.

\begin{theorem}\label{thm:theta alpha iff alpha elem equiv iff alpha equiv}
	Let $(\mc M, M, \ol a)$ and $(\mc N, N, \ol b)$ be pointed systems of the same parameter size, and let $\alpha$ be an ordinal. The following are equivalent.
	\begin{enumerate}[label=(\arabic*)]
		\item\label{item:theta alpha; thm:theta alpha iff alpha elem equiv iff alpha equiv}
			$\mc M, M \vD \theta^{\mc N, N, \ol b}_{\alpha} [\ol{a}]$.
		\item\label{item:alpha equiv; thm:theta alpha iff alpha elem equiv iff alpha equiv}
			$\Rank(M, \ol a, N, \ol b) \geq \alpha \text{ in }\MEF(\mc M, \mc N)$.
		\item\label{item:alpha bis; thm:theta alpha iff alpha elem equiv iff alpha equiv}
			$(\mc M, M, \ol a) \bis_{\alpha} (\mc N, N, \ol b)$.
		\item\label{item:elem equiv; thm:theta alpha iff alpha elem equiv iff alpha equiv}
			$(\mc M, M, \ol a) \equiv^{\dia,\alpha+1}_\infty (\mc N, N, \ol b)$.
	\end{enumerate}
\end{theorem}

\begin{proof}
	$\text{\ref{item:theta alpha; thm:theta alpha iff alpha elem equiv iff alpha equiv}} \Lra \text{\ref{item:alpha equiv; thm:theta alpha iff alpha elem equiv iff alpha equiv}}$ is by \cref{thm:theta alpha iff rank geq alpha}, $\text{\ref{item:alpha equiv; thm:theta alpha iff alpha elem equiv iff alpha equiv}} \Lra \text{\ref{item:alpha bis; thm:theta alpha iff alpha elem equiv iff alpha equiv}}$ is by \cref{lem:alpha equiv iff alpha bis}, and $\text{\ref{item:alpha bis; thm:theta alpha iff alpha elem equiv iff alpha equiv}} \Ra \text{\ref{item:elem equiv; thm:theta alpha iff alpha elem equiv iff alpha equiv}}$ is by \cref{thm:alpha-bisimulation invariance}.

	$\text{\ref{item:elem equiv; thm:theta alpha iff alpha elem equiv iff alpha equiv}} \Ra \text{\ref{item:theta alpha; thm:theta alpha iff alpha elem equiv iff alpha equiv}}$. Note that $(N, \ol b,N, \ol b)$ is winning for Eloise in $\MEF(\mc N, \mc N)$, and hence has rank at least $\alpha$ by \ref{item:infty iff winning; lem:position rank basic results} of \cref{lem:position rank basic results}. Hence by \cref{thm:theta alpha iff rank geq alpha} we have:
	\begin{equation*}
		\mc N, N \vD \theta^{\mc N, N, \ol b}_{\alpha}\left[\ol b\right]
	\end{equation*}
	But $\theta^{\mc N, N, \ol b}_{\alpha}$ has modal-quantifier rank $\alpha$, whence by assumption:
	\begin{equation*}
		\mc M, M \vD \theta^{\mc N, N, \ol b}_{\alpha}[\ol a]\qedhere
	\end{equation*}
\end{proof}

Putting everything together, we can now prove the main result.

\begin{proof}[Proof of \cref{thm:MEF tfae}]
	We have:
	\begin{align*}
		M, \ol a \equiv^\dia_\infty N, \ol b
			&\quad\Lra\quad \forall \alpha \colon M, \ol a \equiv^{\dia,\alpha+1}_\infty N, \ol b \\
			&\quad\Lra\quad \forall \alpha \colon (\mc M, M, \ol a) \bis_{\alpha} (\mc N, N, \ol b) \tag{\cref{thm:theta alpha iff alpha elem equiv iff alpha equiv}} \\
			&\quad\Lra\quad \text{Eloise will win }\MEF(\mc M, M, \ol a, \mc N, N, \ol b) \tag{\cref{lem:position rank basic results}} \\
			&\quad\Lra\quad (\mc M, M, \ol a) \bis (\mc N, N, \ol b) \tag{\cref{thm:Eloise winning iff bis}}
	\end{align*}
\end{proof}

\section{\texorpdfstring{$\Mod(T)$}{Mod(T)} and set-sized potentialist systems}
\label{sec:mod T set}

The remainder of this article will consider following the question.

\begin{question}\label{q:pot sys bis with set}
	When is a potentialist system bitotally bisimilar to one consisting of set-many worlds?
\end{question}

In this section, I focus on those systems of the form $\Mod(T)$ for some first-order $L$-theory $T$, and ask the following specialisation of this question.

\begin{question}\label{q:when mod T bis to set}
	For which first-order theories $T$ is $\Mod(T)$ bitotally bisimilar a set-sized potentialist system?
\end{question}

In this section and the next I make use of basic model theoretic terminology and techniques, such as $\aleph_0$-saturated models and the method of diagrams. For a reference see \cite{HodgesWilfrid1993Mt}.

Let us first see that answer to \cref{q:when mod T bis to set} is not ``all of them'', by considering the theory of directed graphs. Indeed, this potentialist system has a very rich structure; see \cite{hamkins2020modal} for further details.

\begin{proposition}\label{prop:mod directed graphs does not reduce}
	Let $L = \{R\}$ and let $T$ be the $L$-theory of directed graphs. Then $\Mod(T)$ is not bitotally bisimilar with a set-sized system.
\end{proposition}

\begin{proof}
	Using the Fundamental Theorem of Bisimulations (\cref{thm:bisimulation invariance}), it suffices to show that there is a proper class of directed graphs which are pair-wise distinguishable using $L^\dia_\infty$-formulas. In fact we don't need to use any modalities for this.
	
	For every ordinal $\alpha$ the set $(\alpha+1,\in)$ is a model of $T$. The structure of each ordinal can be encoded in an $L_\infty$-sentence as follows. Define the formula $\xi_\alpha(x)$ inductively by:
	\begin{equation*}
		\xi_\alpha(x) \defeq \bigwedge_{\beta < \alpha}\left(\exists y (y R x \wedge \xi_\beta(y))\right) \wedge \forall y \left(y R x \ra\bigvee_{\beta < \alpha} \xi_\beta(y)\right)
	\end{equation*}
	It can be shown by induction that $(\alpha+1,\in) \vD \xi_\beta[\alpha]$ if and only if $\alpha=\beta$. Now let:
	\begin{equation*}
		\nu_\alpha \defeq \exists x(\forall y \neg x R y \wedge \xi_\alpha(x))
	\end{equation*}
	Since every structure $(\alpha+1,\in)$ has a unique top element $\alpha$, we then have that $(\alpha+1 ,\in) \vD \nu_\beta$ if and only if $\alpha=\beta$.
\end{proof}

To provide a positive answer to \cref{q:when mod T bis to set}, we seek a property of first-order theories which ensures that $\Mod(T)$ is not `unboundedly complicated'. The following notion provides a sufficient condition.

\begin{definition}
	Say that a theory $T$ is \emph{$\kappa$-rich} for some cardinal $\kappa \geq \abs L + \aleph_0$ if it has infinite models, and all its models of size $\kappa$ are $\aleph_0$-saturated.
\end{definition}

Let us first see a couple of basic results about this notion.

\begin{lemma}\label{lem:richness monotonic}
	If $T$ is $\kappa$-rich, and $\lam \geq \kappa$, then $T$ is also $\lam$-rich.
\end{lemma}

\begin{proof}
	Indeed, let $M \vD T$ be a $\lam$-model. Take $S \sse M$ finite. By the Downward Löwenheim-Skolem Theorem, there is an elementary substructure $N \preceq M$ of size $\kappa$, containing $S$, which is a model of $T$. Since $T$ is $\kappa$-rich, every type over $S$ is realised in $N$, and hence too in $M$.
\end{proof}

\begin{lemma}\label{lem:every completion kappa cat ra kappa rich}
	Let $T$ be such that every completion is $\kappa$-categorical. Then $T$ is $\kappa$-rich.
\end{lemma}

\begin{proof}
	If all completions are $\kappa$-categorical, then all $\kappa$-models are saturated, so in particular $\aleph_0$-saturated.
\end{proof}

Note that the converse doesn't hold for $\kappa > \abs L$. For example, if $T$ is $\abs{L}$-categorical then it is $\abs{L}^+$-rich, but it need not be $\abs{L}^+$-categorical.

We will see that when $T$ is $\kappa$-rich $\Mod(T)$ is bisimilar with a certain set-sized system. This system is defined by removing all the worlds from $\Mod(T)$ of size greater than $\kappa$, as follows.

\begin{definition}
	Let $\lam$ be a cardinal. Define $\Mod_\lam(T)$ to be the subsystem of $\Mod(T)$ consisting of those models of cardinality less than $\lam$. The system $\ModEm_\lam(T)$ is defined similarly.
\end{definition}

\begin{lemma}
	The renaming bisimulation $\propto \colon \Mod(T) \bis \ModEm(T)$ restricts to a bitotal iso-bisimulation:
	\begin{equation*}
		\Mod_\lam(T) \bis \ModEm_\lam(T)
	\end{equation*}
\end{lemma}

Finally, potentialist systems such as $\Mod(T)$ and $\Mod_\lam(T)$ contain class-many isomorphic copies of each model. Since we are interested in reducing the size of such systems, it will be necessary to consider the equivalent system containing only one model for each isomorphism type.

\begin{definition}
	Let $\mc M$ be a potentialist system. Its \emph{skeleton}, $\Skel(\mc M)$, is the skeleton in the category-theoretic sense, i.e\@. the quotient under the relation which identifies $M$ and $M'$ if there are mutually inverse embeddings between them in $\mc M$. Concretely, for each equivalence class, we pick one representative, and let $\Skel(\mc M)$ be the subsystem of $\mc M$ consisting of these worlds and all embeddings between them.
\end{definition}

\begin{remark}
	Giving $\Skel(\mc M)$ a concrete realisation in general requires some class-level choice, like the Axiom of Global Choice. This can be avoided here, since we will only consider $\Skel(\Mod_\lam(T))$ for some $T$, which can be realised by considering a set models of $T$ whose domain is a subset of some fixed set of size $\lam$.
\end{remark}

\begin{lemma}\label{lem:skel iso bis}
	The relation between $\mc M$ and $\Skel(\mc M)$ which associates each structure to its equivalence class is a bitotal iso-bisimulation.
\end{lemma}

With the preliminaries in place, the main theorem for this section is ready to be proved.

\begin{theorem}\label{thm:kappa rich ra mod T bis mod kappa T}
	If $T$ is $\kappa$-rich then $\Mod(T)$ is bitotally bisimilar with $\Mod_{\kappa^+}(T)$, and thus with a set-sized potentialist system. 
\end{theorem}

\begin{remark}
	Note that $\Mod(T)$ is not \emph{iso}-bisimilar with a set-sized system, for cardinality reasons.
\end{remark}

\begin{proof}[Proof of \cref{thm:kappa rich ra mod T bis mod kappa T}]
	Define $\sim$ between $\ModEm(T)$ and $\ModEm_{\kappa^+}(T)$ as follows. Take $(M, \ol a)$ from $\ModEm(T)$ and $(N, \ol b)$ from $\ModEm_{\kappa^+}(T)$. There are two cases.
	\begin{itemize}
		\item \textbf{Case 1: $\abs M \leq \kappa$}. Put $(M, \ol a) \sim (N, \ol b)$ if and only if $M, \ol a \cong N, \ol b$.
		\item \textbf{Case 2: $\abs M > \kappa$}. Put $(M, \ol a) \sim (N, \ol b)$ if and only if $\abs N = \kappa$ and $\ol a$ and $\ol b$ have the same $L$-types: $\Tp_M(\ol a) = \Tp_N(\ol b)$. For each $(M, \ol a)$, such a $(N, \ol b)$ must exist in $\ModEm_{\kappa^+}(T)$ by the Downward Löwenheim-Skolem Theorem.
	\end{itemize}
	Note that $\sim$ is bitotal. Let us verify that it is a bisimulation. Assume that $(M, \ol a) \sim (N, \ol b)$. First note that, by definition, $M, \ol a \equiv N, \ol b$. Let us now consider condition \ref{item:isomorphism extension; def:bisimulation} on bisimulations. There are two cases.
	
	\textbf{Condition \ref{item:isomorphism extension; def:bisimulation}, case 1: $\abs M \leq \kappa$}. Then $M, \ol a \cong N, \ol b$. Hence for each $c \in M$, there is $d \in N$ such that $(M, \ol a \concat c) \sim (N, \ol b \concat d)$, and conversely.

	\textbf{Condition \ref{item:isomorphism extension; def:bisimulation}, case 2: $\abs M > \kappa$}. Then $\Tp_M(\ol a) = \Tp_N(\ol b)$. By assumption on $T$ and \cref{lem:richness monotonic}, both $M$ and $N$ are $\aleph_0$-saturated. Hence for any $c \in M$ there is $d \in N$ such that $\Tp_N(\ol b \concat d) = \Tp_M(\ol a \concat c)$, and conversely.

	For condition \ref{item:back and forth; def:bisimulation}, there are again two cases. 

	\textbf{Condition \ref{item:back and forth; def:bisimulation}, case 1: $\abs M \leq \kappa$}. Then by definition there is an isomorphism $f \colon (M, \ol a) \to (N, \ol b)$. Assume that there is $\pi \colon M \to M'$ in $\ModEm(T)$. If $\abs{M'} \leq \kappa$, then by definition there is $\rho \colon N \to N'$ in $\ModEm_{\kappa^+}(T)$ such that $f$ extends to an isomorphism $M' \cong N'$. If, on the other hand, $\abs{M'} > \kappa$, then by the Downward Löwenheim-Skolem Theorem, there is $N' \preceq M'$ of cardinality $\kappa$ with $\pi(M) \sse N'$; note that $N' \in \ModEm_{\kappa^+}(T)$. Then we get an embedding $\rho \defeq \pi \circ f^{-1} \colon N \to N'$, with the property that:
	\begin{equation*}
		\Tp_{N'}(\rho(\ol b)) = \Tp_{N'}(\pi(\ol a)) = \Tp_{M'}(\pi(\ol a))
	\end{equation*}
	Whence $(M', \pi(\ol a)) \sim (N', \rho(\ol b))$. See \cref{fig:com diag for B3 case 1; proof:kappa rich ra mod T bis mod kappa T} for a commutative diagram representing the situation. Conversely, any morphism $\rho \colon N \to N'$ in $\ModEm_{\kappa^+}(T)$ immediately gives a morphism $\pi \colon M \to M'$, which is the same up to isomorphism.

	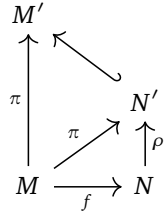
\begin{figure}[b]
		\begin{equation*}
			\begin{tikzcd}
			M'                                                    &                      \\
			                                                      & N' \arrow[lu, hook'] \\
			M \arrow[r, "f"'] \arrow[ru, "\pi"] \arrow[uu, "\pi"] & N \arrow[u, "\rho"']
			\end{tikzcd}
		\end{equation*}
		\caption{The commutative diagram for \ref{item:back and forth; def:bisimulation}, case 1}
		\label{fig:com diag for B3 case 1; proof:kappa rich ra mod T bis mod kappa T}
	\end{figure}

	\textbf{Condition \ref{item:back and forth; def:bisimulation}, case 2: $\abs M > \kappa$}. Take $\pi \colon M \to M'$ in $\ModEm(T)$. By a compactness argument, we can see that that the following is satisfiable (substituting the constants for $\ol b$ from $\Diag(N)$ for the free variables in $\Tp_{M'}(\pi(\ol a))$).
	\begin{equation*}
		T \cup \Diag(N) \cup \Tp_{M'}(\pi(\ol a))(\ol b)
	\end{equation*}
	Therefore, by the method of diagrams, there is a structure $N' \vD T$ of cardinality $\kappa$ with and an embedding $\rho \colon N \to N'$ such that:
	\begin{equation*}
		\Tp_{N'}(\rho(\ol b))=\Tp_{M'}(\pi(\ol a))
	\end{equation*}
	Whence $(M', \pi(\ol a)) \sim (N', \rho(\ol b))$. Conversely, given any morphism $\rho \colon N \to N'$ in $\ModEm_{\kappa^+}(T)$ we can find a corresponding morphism $\pi \colon M \to M'$ by a very similar argument.

	Finally, by \cref{thm:renaming bis total bis} and \cref{lem:skel iso bis}, we have bitotal bisimulations:
	\begin{align*}
		\Mod(T) 
			&\bis \ModEm(T) \\
			&\bis \ModEm_{\kappa^+}(T) \\
			&\bis \Mod_{\kappa^+}(T) \\
			&\bis \Skel(\Mod_{\kappa^+}(T))
	\end{align*}
	the latter of which consists of set-many models.
\end{proof}

\begin{corollary}\label{cor:kappa-cat bis to topped system}
	Let $T$ be $\kappa$-categorical. Then $\Mod(T)$ is bitotally bisimilar with a system consisting of models of size ${<}\kappa$ and a single $\kappa$-model.
\end{corollary}

\begin{proof}
	Note that $\Skel(\Mod_{\kappa^+}(T))$ contains a single $\kappa$-size model.
\end{proof}

\section{\texorpdfstring{$\aleph_0$}{Omega}-categorical theories}
\label{sec:omega categorical}

In this section, I investigate \cref{q:pot sys bis with set} from a different direction. Given a potentialist system $\mc M$, we can consider its \emph{theory} $\Th(\mc M)$: the set of (finitary and non-modal) $L$-sentences which hold in all of its worlds. I first show that when $\Th(\mc M)$ is $\aleph_0$-categorical and has quantifier elimination, then $\mc M$, irrespective its modal structure, is bisimilar to a very small potentialist system. On the other hand, I give an example showing that the result no longer holds when we drop the condition `has quantifier elimination'.

\begin{theorem}\label{thm:count-cat qe}
	Let $\mc M$ be a potentialist system which contains an infinite world, and let $T = \Th(\mc M)$. Assume that $T$ is $\aleph_0$-categorical and has quantifier elimination. Then $\mc M$ is bitotally bisimilar to a system which contains exactly one infinite world, which is countable, and which has no self-embeddings except the identity.

	In particular, when $T$ has no finite models, $\mc M$ is bitotally bisimilar to a singleton system consisting of a countable world with just the identity embedding.
\end{theorem}

\begin{remark}
	In Theorem~16 of \cite{hamkins2020modal} it is shown that if a theory $T$ admits quantifier elimination, then it admits `modality elimination' over $\Mod(T)$: every $\dia\phi \in L^\dia$ is equivalent in $\Mod(T)$ to $\phi$. When $T$ is $\aleph_0$-categorical, this can be seen as a special case of our \cref{thm:count-cat qe}, using the Fundamental Theorem of Bisimulations (\cref{thm:bisimulation invariance}). Indeed, if $M$ is a world in a potentialist system which sees only the identity embedding, then for every $\phi \in L^\dia_\infty$ and $\ol a$ in $M$ we have that $\mc M, M \vD \phi[\ol a]$ if and only if $\mc M, M \vD \dia\phi[\ol a]$. 
\end{remark}

\begin{proof}[Proof of \cref{thm:count-cat qe}]
	Define the system $\mc N$ as follows. Start with all the finite worlds of $\mc M$ together with the embeddings between them. Let $Q$ be a countable model of $T$. Add $Q$ to $\mc N$, together with the identity embedding. Furthermore, for every finite world $M$ in $\mc M$ and embedding $\pi \colon M \to M'$ into an infinite world, by the Downward Lowenheim-Skolem Theorem we can realise $\pi$ as an embedding $\pi^* \colon M \to Q$. Add $\pi^*$ into $\mc N$. Note that multiple embeddings may result in the same starred versions which are added to $\mc N$.

	Now define the relation $\sim$ between $\mc M$ and $\mc N$ from two cases (c.f.\@ the proof of \cref{thm:kappa rich ra mod T bis mod kappa T}).
	\begin{itemize}
		\item \textbf{Case 1: $M$ is finite}. Relate $(M, \ol a)$ and $(N, \ol b)$ if and only if they are equal.
		\item \textbf{Case 2: $M$ is infinite}. Relate $(M, \ol a)$ and $(N, \ol b)$ if and only if $N=Q$ and $\ol a$ and $\ol b$ agree on all atomic formulas: $\Delta_0^M(\ol a) = \Delta_0^N(\ol b)$.
	\end{itemize}
	It is not hard to see using the Downward Lowenheim-Skolem Theorem that this is a bitotal relation. Let us verify that it is a bisimulation. Assume that $(M, \ol a) \sim (N, \ol b)$.
	\begin{itemize}
		\item[\ref{item:partial isomorphism; def:bisimulation}] \label{item:partial isomorphism; proof:qe principal types modality elimination} 
			The first condition follows by definition.

		\item[\ref{item:isomorphism extension; def:bisimulation}] \label{item:isomorphism extension; proof:qe principal types modality elimination}
			When $M$ is finite this condition is immediate, so assume that $M$ is infinite. By the Ryll-Nardzewski Theorem applied to $\Th(M)$, the type of every tuple is principal. Moreover, since $T$ has quantifier elimination, the complete types of $\ol a$ and $\ol b$ are the same. Hence any extension of one can be mirrored in the other.

		\item[\ref{item:back and forth; def:bisimulation}] \label{item:back and forth; proof:qe principal types modality elimination} 
			There are several cases to consider. If $M$ is infinite, then since any embedding preserves atomic types, any modal extension on one side can be mirrored on the other. So we can assume that $M$ is finite (and hence $N = M$). Now any embedding from a finite world to a finite world on on side can be identically mirrored on the other, since $\mc N$ includes all embeddings from $\mc M$ between finite worlds. So we are left with two cases.
			\begin{enumerate*}[label=(\roman*)]
				\item If we have $\pi \colon M \to M'$ in $\mc M$ with $M'$ infinite, then as above we can realise this as $\pi^* \colon M \to Q$. Note that the atomic types of $\pi(\ol a)$ and $\pi^*(\ol b)$ are the same.
				\item If we have $\pi^* \colon N \to Q$ in $\mc N$, then by definition there is $\pi \colon M \to M'$ in $\mc M$ from which $\pi^*$ was obtained. Again the atomic types of $\pi(\ol a)$ and $\pi^*(\ol b)$ are the same.\qedhere
			\end{enumerate*}
	\end{itemize}
\end{proof}

\begin{corollary}
	Let $T$ be an $\aleph_0$-categorical first-order $L$-theory which has quantifier elimination and no finite models. Then between any two potentialist systems which consist of models of $T$ there is a bitotal bisimulation.
\end{corollary}

A partial converse can be obtained to \cref{thm:count-cat qe}. In order to make it work, we need to make sure that $\mc M$ has `enough structure' from $\Mod(T)$.

\begin{theorem}\label{thm:iso to pot sys one inf model}
	Let $\mc M$ be a potentialist system and let $T = \Th(\mc M)$. Assume that $\mc M$ has a isomorphic copy of every countable model of $T$, together with all embeddings between these isomorphic copies. Assume further that $\mc M$ is bitotally bisimilar to a system $\mc N$ which contains exactly one infinite world which has no self-embeddings except the identity. Then $T$ is $\aleph_0$-categorical and has quantifier elimination.
\end{theorem}

Note that we don't need the infinite world in $\mc N$ to be countable.

\begin{proof}
	Let $Q$ be the unique infinite world in $\mc N$, and let $\sim$ be the bisimulation between $\mc M$ and $\mc N$. Note that if $M$ in $\mc M$ is infinite, then every pair $(M, \ol a)$ is related to a pair with $Q$.

	First, let us see that $T$ is $\aleph_0$-categorical. Take $X, Y$ any countable models of $T$. We can assume without loss of generality that $X$ and $Y$ are in $\mc M$. We will build up an isomorphism $X \cong Y$ using a back-and-forth argument. First, using that $\sim$ is bitotal, for any $x_0 \in X$ the pair $(X, (x_0))$ is related to some pair $(Q, (q_0))$, which must in turn be related some pair $(Y, (y_0))$. Going the other direction, for any $y_1 \in Y$, the pair $(Y, (y_0, y_1))$ is related to some $(Q, (q_0, q_1))$, which is related to some $(X, (x_0, x_1))$. Proceeding in this way, using enumerations of $X$ and $Y$, we obtain a bijection $X \to Y$. Using condition \ref{item:partial isomorphism; def:bisimulation} on bisimulations, each finite correspondence is a partial isomorphism. Hence the whole correspondence $X \to Y$ is an isomorphism.

	Second, let us check that $T$ admits quantifier elimination. First note that since $T$ is $\aleph_0$-categorical, it is complete. Let $r$ be any atomic type over $T$, and let $p$ and $q$ be extensions of $r$ to a complete type. The goal is to show that $p=q$. Let $W \in \mc M$ be a countable model with $\ol w$ in $W$ realising $r$. Since $\sim$ is bitotal, there is a pair $(Q,\ol q)$ related to $(W, \ol w)$. By a compactness argument, using that $T$ is complete, there are countable models $X$ and $Y$ of $T$ and embeddings $\pi \colon W \to X$ and $\rho \colon W \to Y$ such that $\pi(\ol w)$ realises $p$ and $\rho(\ol w)$ realises $q$. We can assume that $X, Y, \pi, \rho$ are in $\mc M$. Now, since in $\mc N$ the only embedding from $Q$ is the identity, we have that:
	\begin{equation*}
		(X, \rho(\ol w)) \sim (Q, \ol q) \sim (Y, \rho(\ol w))
	\end{equation*}
	By the Fundamental Theorem of Bisimulations (\cref{thm:bisimulation invariance}), in particular:
	\begin{equation*}
		X, \rho(\ol w) \equiv Y, \rho(\ol w)
	\end{equation*}
	Hence $p=q$ as required.
\end{proof}

Let us see that it is necessary to assume that $\mc M$ has a isomorphic copy of every countable model of $T$. Let $T$ be the theory of non-zero $\Q$-vector spaces, and let $\mc M$ be the subsystem of $\Mod(T)$ consisting of the uncountable worlds. Then $\mc M$ is bitotally bisimilar with the system consisting just of the $\aleph_0$-dimensional $\Q$-vector space with the identity embedding. However $T = \Th(\mc M)$ is not $\aleph_0$-categorical.

I conclude this section with an example of how things can go wrong for \cref{thm:count-cat qe} when we don't require that $T$ admit quantifier elimination. The point is that once we are allowed to play with slightly more intricate first-order models, we can start arranging them to produce a rich modal structure.

Let $L = \{\approx\}$, and let $T$ specify that $\approx$ is an equivalence relation with infinitely many size-$1$ and size-$2$ equivalence classes, and no other classes. Then $T$ is $\aleph_0$-categorical, but does not have quantifier elimination. However every one-variable formula is equivalent to either an existential or a universal formula.

We construct for every ordinal $\alpha$ a pointed system $(\mc M_\alpha, M_\alpha)$ of countable models of $T$. See \cref{fig:pot system M3; ex:omega-cat non-reducible pot sys} for a representation of the system $\mc M_3$. The underlying modal structure of $\mc M_\alpha$ is the result of viewing $\alpha$ as a well-founded tree $X_\alpha$. (That is, $X_\alpha$ consists of a root sitting below disjoint copies of each $X_\beta$ for $\beta < \alpha$.) We place countable models of $T$ along this structure as follows. First let $N$ be any countable model. At the root of $X_\alpha$ we place a model consisting of $N$ plus new elements $u_0, u_1, \ldots$, each in a singleton equivalence class. On the next level up we place disjoint copies of the same model with a new element $v_0$, which appears in an equivalence class with $u_0$. The model at the root embeds into those above it via the inclusion embedding. The construction proceeds upwards in this fashion. Let $M_\alpha$ be the model appearing at the root in this system.

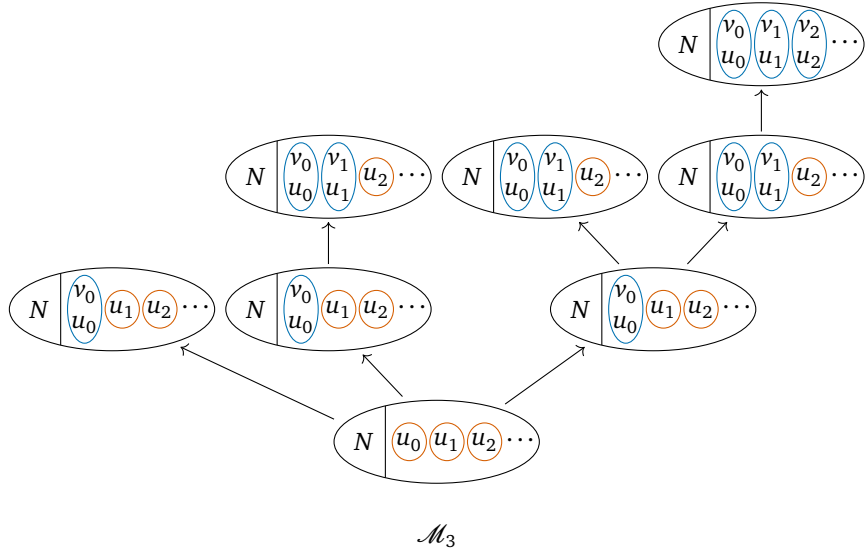
\begin{figure}
	\begin{equation*}
		\begin{tikzpicture}[xscale=0.45,yscale=0.5]
			\begin{scope}
				\draw (0,0) ellipse (3 and 1.1);
				\clip (0,0) ellipse (3 and 1.1);
				\draw (-1.5,-1) -- (-1.5,1);
				\node at (-2.1,0) {$N$};
				\node[black] at (-0.8,0) {$u_0$};
				\draw[pallette_vermillion] (-0.8,0) circle (0.5);
				\node[black] at (0.3,0) {$u_1$};
				\draw[pallette_vermillion] (0.3,0) circle (0.5);
				\node[black] at (1.4,0) {$u_2$};
				\draw[pallette_vermillion] (1.4,0) circle (0.5);
				\node at (2.5,0) {$\cdots$};
				\coordinate (b000) at (0,-1.2);
				\coordinate (t000) at (0,1.2);
			\end{scope}
			\begin{scope}[xshift=-270,yshift=100]
				\draw (0,0) ellipse (3 and 1.1);
				\clip (0,0) ellipse (3 and 1.1);
				\draw (-1.5,-1) -- (-1.5,1);
				\node at (-2.1,0) {$N$};
				\node[black] at (-0.8,-0.4) {$u_0$};
				\node[black] at (-0.8,0.4) {$v_0$};
				\draw[pallette_blue] (-0.8,0) ellipse (0.5 and 0.9);
				\node[black] at (0.3,0) {$u_1$};
				\draw[pallette_vermillion] (0.3,0) circle (0.5);
				\node[black] at (1.4,0) {$u_2$};
				\draw[pallette_vermillion] (1.4,0) circle (0.5);
				\node at (2.5,0) {$\cdots$};
				\coordinate (b001) at (0,-1.2);
				\coordinate (t001) at (0,1.2);
			\end{scope}
			\begin{scope}[xshift=-90,yshift=100]
				\draw (0,0) ellipse (3 and 1.1);
				\clip (0,0) ellipse (3 and 1.1);
				\draw (-1.5,-1) -- (-1.5,1);
				\node at (-2.1,0) {$N$};
				\node[black] at (-0.8,-0.4) {$u_0$};
				\node[black] at (-0.8,0.4) {$v_0$};
				\draw[pallette_blue] (-0.8,0) ellipse (0.5 and 0.9);
				\node[black] at (0.3,0) {$u_1$};
				\draw[pallette_vermillion] (0.3,0) circle (0.5);
				\node[black] at (1.4,0) {$u_2$};
				\draw[pallette_vermillion] (1.4,0) circle (0.5);
				\node at (2.5,0) {$\cdots$};
				\coordinate (b010) at (0,-1.2);
				\coordinate (t010) at (0,1.2);
			\end{scope}
			\begin{scope}[xshift=-90,yshift=200]
				\draw (0,0) ellipse (3 and 1.1);
				\clip (0,0) ellipse (3 and 1.1);
				\draw (-1.5,-1) -- (-1.5,1);
				\node at (-2.1,0) {$N$};
				\node[black] at (-0.8,-0.4) {$u_0$};
				\node[black] at (-0.8,0.4) {$v_0$};
				\draw[pallette_blue] (-0.8,0) ellipse (0.5 and 0.9);
				\node[black] at (0.3,-0.4) {$u_1$};
				\node[black] at (0.3,0.4) {$v_1$};
				\draw[pallette_blue] (0.3,0) ellipse (0.5 and 0.9);
				\node[black] at (1.4,0) {$u_2$};
				\draw[pallette_vermillion] (1.4,0) circle (0.5);
				\node at (2.5,0) {$\cdots$};
				\coordinate (b011) at (0,-1.2);
				\coordinate (t011) at (0,1.2);
			\end{scope}
			\begin{scope}[xshift=180,yshift=100]
				\draw (0,0) ellipse (3 and 1.1);
				\clip (0,0) ellipse (3 and 1.1);
				\draw (-1.5,-1) -- (-1.5,1);
				\node at (-2.1,0) {$N$};
				\node[black] at (-0.8,-0.4) {$u_0$};
				\node[black] at (-0.8,0.4) {$v_0$};
				\draw[pallette_blue] (-0.8,0) ellipse (0.5 and 0.9);
				\node[black] at (0.3,0) {$u_1$};
				\draw[pallette_vermillion] (0.3,0) circle (0.5);
				\node[black] at (1.4,0) {$u_2$};
				\draw[pallette_vermillion] (1.4,0) circle (0.5);
				\node at (2.5,0) {$\cdots$};
				\coordinate (b100) at (0,-1.2);
				\coordinate (t100) at (0,1.2);
			\end{scope}
			\begin{scope}[xshift=90,yshift=200]
				\draw (0,0) ellipse (3 and 1.1);
				\clip (0,0) ellipse (3 and 1.1);
				\draw (-1.5,-1) -- (-1.5,1);
				\node at (-2.1,0) {$N$};
				\node[black] at (-0.8,-0.4) {$u_0$};
				\node[black] at (-0.8,0.4) {$v_0$};
				\draw[pallette_blue] (-0.8,0) ellipse (0.5 and 0.9);
				\node[black] at (0.3,-0.4) {$u_1$};
				\node[black] at (0.3,0.4) {$v_1$};
				\draw[pallette_blue] (0.3,0) ellipse (0.5 and 0.9);
				\node[black] at (1.4,0) {$u_2$};
				\draw[pallette_vermillion] (1.4,0) circle (0.5);
				\node at (2.5,0) {$\cdots$};
				\coordinate (b101) at (0,-1.2);
				\coordinate (t101) at (0,1.2);
			\end{scope}
			\begin{scope}[xshift=270,yshift=200]
				\draw (0,0) ellipse (3 and 1.1);
				\clip (0,0) ellipse (3 and 1.1);
				\draw (-1.5,-1) -- (-1.5,1);
				\node at (-2.1,0) {$N$};
				\node[black] at (-0.8,-0.4) {$u_0$};
				\node[black] at (-0.8,0.4) {$v_0$};
				\draw[pallette_blue] (-0.8,0) ellipse (0.5 and 0.9);
				\node[black] at (0.3,-0.4) {$u_1$};
				\node[black] at (0.3,0.4) {$v_1$};
				\draw[pallette_blue] (0.3,0) ellipse (0.5 and 0.9);
				\node[black] at (1.4,0) {$u_2$};
				\draw[pallette_vermillion] (1.4,0) circle (0.5);
				\node at (2.5,0) {$\cdots$};
				\coordinate (b110) at (0,-1.2);
				\coordinate (t110) at (0,1.2);
			\end{scope}
			\begin{scope}[xshift=270,yshift=300]
				\draw (0,0) ellipse (3 and 1.1);
				\clip (0,0) ellipse (3 and 1.1);
				\draw (-1.5,-1) -- (-1.5,1);
				\node at (-2.1,0) {$N$};
				\node[black] at (-0.8,-0.4) {$u_0$};
				\node[black] at (-0.8,0.4) {$v_0$};
				\draw[pallette_blue] (-0.8,0) ellipse (0.5 and 0.9);
				\node[black] at (0.3,-0.4) {$u_1$};
				\node[black] at (0.3,0.4) {$v_1$};
				\draw[pallette_blue] (0.3,0) ellipse (0.5 and 0.9);
				\node[black] at (1.4,-0.4) {$u_2$};
				\node[black] at (1.4,0.4) {$v_2$};
				\draw[pallette_blue] (1.4,0) ellipse (0.5 and 0.9);
				\node at (2.5,0) {$\cdots$};
				\coordinate (b111) at (0,-1.2);
				\coordinate (t111) at (0,1.2);
			\end{scope}
			\draw[->] ($(t000) + (-3,-0.6)$) -- ($(b001) + (2,0.2)$);
			\draw[->] ($(t000) + (-1,0)$) -- ($(b010) + (1,0)$);
			\draw[->] ($(t000) + (2,-0.2)$) -- ($(b100) + (-2,0.2)$);
			\draw[->] (t010) -- (b011);
			\draw[->] ($(t100) + (-1,0)$) -- ($(b101) + (1,0)$);
			\draw[->] ($(t100) + (1,0)$) -- ($(b110) + (-1,0)$);
			\draw[->] (t110) -- (b111);
			\node at (0,-2.5) {$\mc M_3$};
		\end{tikzpicture}
	\end{equation*}
	\caption{The potentialist system $\mc M_3$}
	\label{fig:pot system M3; ex:omega-cat non-reducible pot sys}
\end{figure}

\begin{theorem}
	Each $(\mc M_\alpha, M_\alpha)$ satisfies a different $L^\dia_\infty$-theory. Hence the system consisting of the disjoint union of each $\mc M_\alpha$ cannot be bitotally bisimilar with a set-sized system.
\end{theorem}

\begin{proof}
	The idea is that each pair $(u_m,v_m)$ acts as a `button'. The lowest model has all $\aleph_0$-many buttons unpushed. Each time we move up a level, we push an extra button (i.e. add the element $v_m$ to make a $2$-element equivalence class). Of course, each model is isomorphic, so the first-order structure cannot tell us how many buttons have been pushed, or how far up the tree we are. However, using the modal language, we can indeed talk about pushing a button, to get to a higher level.

	Let:
	\begin{gather*}
		p(x) \defeq \neg \exists y (x \neq y \wedge x \approx y) \\
		q(x) \defeq \exists y (x \neq y \wedge x \approx y)
	\end{gather*}
	Define for each $\alpha$ the sentence $\theta_\alpha$ by induction:
	\begin{equation*}
		\theta_\alpha \defeq \exists x \left(p(x) \wedge \bigwedge_{\beta < \alpha} \dia (q(x) \wedge \theta_\beta)\right) \wedge \forall x \left(p(x) \ra \bx \left(q(x) \ra \bigvee_{\beta < \alpha}\theta_\beta\right)\right)
	\end{equation*}
	In words: ``there is a button which is pushed in future models satisfying $\theta_\beta$ for all $\beta < \alpha$, and whenever a button gets pushed in the future, that model satisfies $\theta_\beta$ for some $\beta < \alpha$''. By induction we see that:
	\begin{equation*}
		\mc M_\beta, M_\beta \vD \theta_\alpha \quad \Lra \quad \beta=\alpha
	\end{equation*}
	Therefore, each $(\mc M_\alpha, M_\alpha)$ satisfies a different $L^\dia_\infty$-theory.

	Now, let $\mc M$ be the disjoint union of every $\mc M_\alpha$. If this were bitotally bisimilar with a set-sized system $\mc N$, then class-many of the $M_\alpha$'s would be related to the same world in $N$. But this is impossible by the Fundamental Theorem of Bisimulations (\cref{thm:bisimulation invariance}).
\end{proof}

\section{Open questions}
\label{sec:conclusion}

The first part of this article presented a number of foundational results concerning bisimulations of potentialist systems. The second part considered \cref{q:pot sys bis with set}: when is a potentialist system $\mc M$ bisimilar with a set-sized one? In \cref{sec:mod T set}, I showed that if $T$ is $\kappa$-rich for some $\kappa$ then $\Mod(T)$ is bisimilar with a set-sized system. Does the converse hold?  An iteration argument shows that if $\Mod(T)$ is bisimilar with a set-sized system, is it necessarily bisimilar with some $\Mod_{\lam}(T)$. Will it further be bisimilar with some $\Mod_{\kappa^+}(T)$?

\Cref{sec:omega categorical} seeks to give answers to \cref{q:pot sys bis with set} based solely on $\Th(\mc M)$. This is a rather crude measure, and takes no account of the modal structure of $\mc M$. \Cref{thm:count-cat qe} shows that if $\Th(\mc M)$ is $\aleph_0$-categorical and admits quantifier elimination, then $\mc M$ is bisimilar with a set-sized system. Is this the best result possible? If instead of $\Th(\mc M)$ we consider the set of (finitary) first-order modal sentences satisfied by $\mc M$, can we obtain a more general characterisation?

\section{Acknowledgements}

I would like to thank my supervisor, Joel David Hamkins, for helpful guidance and proof-reading throughout the process of writing this article. This work was supported by the EPSRC [studentship with project reference \emph{2271793}].

	\printbibliography

\end{document}